\documentclass{article}
\usepackage{amsmath,color,amsfonts,amsthm}
\usepackage{amssymb,enumerate,verbatim}
\usepackage[shortlabels]{enumitem}
\usepackage{hyperref}
\usepackage[bottom,hang,flushmargin]{footmisc}
\usepackage{tikz}
\usetikzlibrary{calc}

\usepackage{a4wide}
\usepackage{graphicx}
\newtheorem{lemma}{Lemma}[section]
\newtheorem{corollary}[lemma]{Corollary}
\newtheorem{claim}[lemma]{Claim}
\newtheorem{theorem}[lemma]{Theorem}

\newtheorem{definition}[lemma]{Definition}
\newtheorem{conjecture}[lemma]{Conjecture}

\date{}

\newcommand{\eps}{\ensuremath{\varepsilon}}

\newcommand{\Prob}{\mathbb{P}}
\newcommand{\E}{\mathbb{E}}

 \newcommand{\Alexey}[1]{}

\author{Peter Keevash\thanks{Mathematical Institute, University of Oxford, Oxford, UK. 
E-mail: keevash@maths.ox.ac.uk.
Research supported in part by ERC Consolidator Grant 647678.}
\and Alexey Pokrovskiy\thanks{Department of Economics, Mathematics, and Statistics, Birkbeck College,  London. dr.alexey.pokrovskiy@gmail.com.}
\and Benny Sudakov\thanks{Department of Mathematics, ETH, 8092 Z\"urich, Switzerland, e-mail: benny.sudakov@gmail.com. 
Research supported in part by SNSF grant 200021-175573.}
\and Liana Yepremyan\thanks{Department of Mathematics, Statistics, and Computer Science, University of Illinois at Chicago, Chicago, USA,
London School of Economics, Department of Mathematics, London, UK,
e-mail:lyepre2@uic.edu, l.yepremyan@lse.ac.uk, 
Research supported by Marie Sklodowska Curie Global Fellowship, H2020-MSCA-IF-2018:846304}}

\title{\vspace{-0.9cm}New bounds for Ryser's conjecture and related problems}

\begin{document}
\maketitle
\begin{abstract}
{A Latin square of order $n$ is an $n \times n$ array filled with $n$ symbols such that each symbol appears only once in every row or column and a transversal is a collection of cells which do not share the same row, column or symbol.  
The study of Latin squares goes back more than 200 years to the work of Euler. One of the most famous open problems in this area is a conjecture of
Ryser-Brualdi-Stein from 60s which says that every Latin square  of order $n\times n$ contains a transversal of order $n-1$. In this paper we prove the existence 
of a transversal of order $n-O(\log{n}/\log{\log{n}})$, improving the celebrated bound of $n-O(\log^2n)$ by Hatami and Shor. Our approach (different from that of Hatami-Shor) is quite general and gives several other applications as well.
We obtain a  new lower bound on a 40 year old conjecture of Brouwer on the maximum matching in Steiner triple systems, showing that every such system of order $n$ is guaranteed to have a matching of size $n/3-O(\log{n}/\log{\log{n}})$. 
This substantially improves the current best result of Alon, Kim and Spencer which has the error term of order $n^{1/2+o(1)}$. Finally,  
we also show that $O(n\log{n}/\log{\log{n}})$ many symbols in Latin arrays suffice to guarantee a full transversal, improving on previously known bound of $n^{2-\eps}$.
The proofs combine in a novel way the semirandom method together with the robust expansion properties of edge-coloured pseudorandom graphs to show the existence of a rainbow matching covering all but $O(\log n/\log{\log{n}})$ vertices.
All previous results, based on the semi-random method, left uncovered at least $\Omega(n^{\alpha})$ (for some constant $\alpha>0$) vertices.}    
\end{abstract}

\section{Introduction}
A Latin square of order $n$ is an $n\times n$ array filled with $n$ symbols so that every
symbol appears only once in each row and in each column. A  transversal is a collection of cells of the Latin square which do not share the same row, column or symbol. A full transversal is a transversal of order $n$.
The study of Latin squares goes back to the work of Euler \cite{euler1782recherches} in 18th century, who asked a question equivalent to ``for which $n$ is there an $n\times n$ Latin square which can be decomposed into $n$ disjoint full transversals?''.  Well known examples of Latin squares are multiplication tables of finite groups. Latin squares have connections to $2$-dimensional permutations, design theory, finite projective planes and error correcting codes. 

It is easy to see that there are many Latin squares without full transversals (for example the addition table of the group $\mathbb{Z}_{4}$)
and it is a hard problem to determine when full transversals exist. This question is very difficult even in the case of multiplication tables of finite groups.
In 1955 Hall and Paige \cite{hall1955complete} conjectured that the multiplication table of a group $G$ has a full transversal exactly if the $2$-Sylow subgroups of $G$ are trivial or non-cyclic.
It took 50 years to establish this conjecture and its proof is based on the classification of finite simple groups (see \cite{wilcox2009reduction} and the references therein). Very recently an alternative proof of this conjecture was found for large groups using tools from analytic number theory~\cite{eberhard2020asymptotic}. 
The most famous open problem on transversals in general Latin squares is the following conjecture of Ryser, Brualdi and Stein~\cite{ryser1967neuere, stein1975transversals, brualdi1991combinatorial}.
 
\begin{conjecture}
\label{prob-partial-ryser-brualdi-stein}
Every $n\times n$ Latin square has a transversal of order $n-1$. Moreover if $n$ is odd it has a full transversal. 
\end{conjecture}

Most research towards the Ryser-Brualdi-Stein conjecture has focused on proving that all $n\times n$ Latin squares have large transversals (trying to get as close to $n-1$ as possible). Here Koksma~\cite{koksma1969lower} found transversals of size $2n/3+O(1)$ and Drake~\cite{drake1977maximal} improved this to $3n/4+O(1)$. The first asymptotic proof of the conjecture was obtained by Brouwer, De Vries, and Wieringa~\cite{brouwer1978lower} and independently by Woolbright~\cite{woolbright1978n} who found transversals of size $n-\sqrt{n}$. This was improved 
in 1982 by Shor~\cite{shor1982} to $n-O(\log^2 n)$. His paper had a mistake which was later rectified, using the original approach, by Hatami and Shor \cite{hatami2008lower}. For the last, nearly forty years, this was the 
best known bound for the Ryser-Brualdi-Stein conjecture. Our first theorem improves this result as follows.

\begin{theorem}\label{main-ryser}  There exist a constant $k$  such that every $n\times n$ Latin square  contains a transversal of order $n- k\frac{\log{n}}{\log{\log{n}}}$.
\end{theorem}

A Latin array is $n\times n$ square filled with an  arbitrary number of symbols such that  no symbol appears twice in the same row or column. Latin arrays are natural extensions of Latin squares, and  also been extensively studied.
A familiar example of such an array is a multiplication table between elements of two subsets of equal size in some group. It is generally believed that extra symbols in a Latin array should help to find  full transversals.
Motivated by this Akbari and Alipour \cite{akbari2004transversals} conjectured that any Latin array of order $n$
with at least $n^2/2$ different symbols contains a full transversal. Progress towards this conjecture was independently obtained by
Best, Hendrey, Wanless, Wilson and Wood \cite{best2018transversals}
(who showed that $(2-\sqrt{2})n^2$ symbols suffice)
and Bar\'at and Nagy \cite{barat2017transversals} 
(who showed that $3n^2/4$ symbols suffice). Very recently Montgomery, Pokrovskiy, Sudakov~\cite{montgomery2019decompositions} and  Keevash, Yepremyan~\cite{keevash2020number} independently showed that $n^{2-\eps}$ many symbols suffice to  guarantee a full transversal.
Here we substantially improve these results. 

\begin{theorem}\label{main-manysymbols}  There exist a constant $k$ such that every $n\times n$ Latin array filled with $kn\log{n}/\log{\log{n}}$ many symbols contains a full transversal.
\end{theorem}

\noindent
It is worth pointing out that the problem of Akbari and Alipour is closely related to finding transversals in Latin squares, namely Conjecture \ref{prob-partial-ryser-brualdi-stein}. In particular, the last theorem implies 
Theorem \ref{main-ryser}. Indeed, start with an $n \times n$ Latin square and then substitute distinct new symbols in the first $k\log{n}/\log{\log{n}}$ rows, such that every symbol is used only once. Then Theorem \ref{main-manysymbols} guarantees us a full transversal. Since this transversal can use at most $k\log{n}/\log{\log{n}}$ cells from the first $k\log{n}/\log{\log{n}}$ rows, upon removing these we are left with a transversal of the original Latin square which has size $n-k\log{n}/\log{\log{n}}$. 

All the above results and problems can be rephrased as statements about matchings in hypergraphs. To see this, we construct from an $n\times n$ Latin square $L$
the following  $3$-uniform hypergraph $\mathcal H$ on $3n$ vertices.  The vertices of $\mathcal{H}$ are $V(\mathcal H)=R\cup C\cup S$ where $R$ are the rows of $L$, $C$ the columns of $L$, and $S$ the symbols of $L$. There is an edge in $\mathcal{H}$ for every entry of $L$. If the $(i,j)$-th entry of $L$ has symbol $s$, then $\{i,j,s\}$ is a hyperedge of $\mathcal{H}$. It is easy to check that under this transformation, the hypergraph we obtain is $n$-regular, there is exactly one edge containing a given pair of vertices,  and that transversals in $L$ correspond to matchings in $\mathcal H$.

The problem of finding nearly perfect matchings in regular hypergraphs has a long history in discrete mathematics and such results have many applications to other problems as well. For example R\"odl~\cite{rodl1985packing} proved the Erd\H{o}s-Hanani Conjecture on existence of approximate designs by essentially showing that regular hypergraphs with bounded codegrees have nearly-perfect matchings. This paper introduced the celebrated technique of ``R\"odl's nibble'' which is a versatile approach for finding large matchings in hypergraphs in semi-random manner. 
One famous example of a regular hypergraph with bounded codegress is a Steiner triple system, which is a $3$-uniform hypergraph on $n$ vertices in which every pair of vertices is in a unique edge.
The existence of such triple systems was established by Kirkman in 1847. By definition, this hypergraph is $(n-1)/2$-regular and has all codegrees equal to one. 
The problem of existence of large matchings in Steiner triple systems was posed about forty years ago by Brouwer~\cite{brouwer1981size}.

 \begin{conjecture}{\label{kahn-conjecture}} Every Steiner triple system of order $n$ contains a matching of size $(n-4)/3$.
 \end{conjecture}
Over the years this conjecture attracted a lot of attention. Wang~\cite{wang1978self} showed that  every Steiner triple system has a matching of size $2n/9-O(1)$. Lindner  and Phelps~\cite{lindner1978note} found a matching of size $4n/15-O(1)$. 
Brouwer~\cite{brouwer1981size} obtained the first asymptotic result by finding matchings of size $n/3-O(n^{2/3})$. Using a clever refinement of R\"odl's nibble combined with large deviation inequalities, Alon, Kim, and Spencer~\cite{alon1997nearly} obtained the best current bound. They show the existence of a matching covering all but  $O(n^{1/2}\log^{3/2} n)$ vertices. Here we improve this twenty year old result and obtain the first sub-polynomial upper bound on the number of vertices uncovered by the maximum matching. 
\begin{theorem}\label{main-steiner} There is a constant $k$ such that every Steiner triple system $S$ on $n$ vertices has a  matching  of size at least $n/3-k\log{n}/\log{\log{n}}$.
\end{theorem}
Our methods combine in a novel way the R\"odl's nibble together with the robust expansion properties of edge-coloured pseudorandom graphs and apply in far more general settings than any of the above theorems and conjectures. The main technical theorem we prove can be used to show that $3$-uniform hypergraphs satisfying certain pseudorandomness properties have a  matching covering all but $O(\log n /\log{\log{n}})$ vertices. All previous comparable theorems left $n^{\alpha}$ vertices uncovered.

\subsection{Proof ideas}\label{Section_proof_sketch}

\subsubsection*{Coloured graphs and rainbow matchings}
Although our main results are about transversals in Latin arrays/squares and matchings in hypergraphs, all our proofs will take place in a different setting. This will be the setting of finding \emph{rainbow matchings} in \emph{properly edge-coloured} complete bipartite graphs. Recall that a \emph{proper} edge-colouring of a graph is one where all edges incident to the same vertex have different colours. A  matching in a coloured graph is rainbow  if all its edges have different colours.  A \emph{linear} hypergraph is a hypergraph in which every pair of vertices lies in at most one edge. In this paper we will use extensively that the following three kinds of objects are equivalent:
\begin{itemize}
\item An $n\times n$ Latin array filled wth $m$ symbols.
\item A linear $3$-partite, $3$-uniform hypergraph with partition sizes $(n,n,m)$.
\item A properly edge-coloured complete bipartite graph $K_{n,n}$ with $m$ colours.
\end{itemize} 

The connection between Latin arrays and linear hypergraphs was already described in the introduction. To see the reduction to coloured graphs consider an $n\times n$ Latin array $L$ filled with $m$ symbols. Using it we can construct the following  proper edge-colouring of  $K_{n,n}$. Label the vertices of $K_{n,n}$ by $\{x_1, \dots, x_n, y_1, \dots, y_n\}$, and join $x_i$ to $y_j$ with a colour $\ell$ edge whenever the $ij$-th  entry of $L$ is $\ell$. This is a proper edge-colouring with $m$ colours due to the properties of Latin arrays. A size $t$ transversal in the Latin array corresponds to a rainbow matching with $t$ edges in $K_{n,n}$. Note that in case of Latin squares we have $m=n$ in the above statement. Thus under this transformation, Theorem~\ref{main-ryser} is equivalent to the following.

\begin{theorem}\label{Theorem_Ryser_Rainbow_Version}
There exists a constant $k$ such that every properly $n$-edge-coloured $K_{n,n}$ has a rainbow matching of size $n-k\frac{\log n}{\log \log n}$.
\end{theorem}

\noindent
Similarly Theorem~\ref{main-manysymbols} has the following equivalent form.
\begin{theorem}\label{Theorem_ManySymbol_Rainbow_Version}
There exists a constant $k$ such that every properly edge-coloured $K_{n,n}$ with $kn\frac{\log n}{\log \log n}$ colours has a rainbow perfect matching.
\end{theorem}
\noindent
Although, as we already mentioned in the previous section, this theorem can be used to prove Theorem \ref{Theorem_Ryser_Rainbow_Version}, we deduce both of them from a more general result about matchings in properly edge-coloured ``typical" (i.e., both edges and colours have some pseudorandom properties) graphs which we obtain in Section 4.

The reduction of the task of finding large matchings in a Steiner triple system $S$ to a graph problem is slightly more subtle, and the details can be found in Section~6. The main idea is to randomly select a tripartition $(A,B,C)$ of $S$ and consider only the edges that respect this tripartition. The first two parts induce a properly  edge-coloured graph $G$ where we think of the colour of an edge $ab$ with $a\in A,b\in B$ to be $c\in C$ if $abc\in S$. Note that any rainbow matching in $G$ induces a matching of the same size in $S$. The graph $G$ turns out to be typical in this coloured setting we mentioned above, and we can also guarantee $|A|=|B|=|C|=n/3$. Thus from our general result (more precisely, Corollary~\ref{Corollary_typical_matching}) it follows that $G$ contains a rainbow matching  of size $n/3-O(\log{n}/\log{\log{n}})$, and therefore, $S$ has a matching of the same size.

\subsubsection*{R\"odl Nibble and expansion}
\label{subsec:nibble}
R\"odl introduced a method called ``R\"odl's nibble", which can be used to find matchings in a wide variety of settings. In particular it applies in the setting of Theorem~\ref{Theorem_Ryser_Rainbow_Version} to give a rainbow matching of size $n-O(n^{1-\alpha})$ (for some small constant $\alpha>0$). Our ideas very much build on this result. At a high level, our proof consists of starting with a matching produced by R\"odl's nibble and then modifying it to get a matching of size $n-O(\log n/\log\log n)$. Although our methods apply for all coloured pseudorandom graphs let us demonstrate its main ideas for the simplest case, $K_{n,n}$.

The basic idea of R\"odl's nibble is to construct a matching in several steps, each time taking a collection of random edges. To imitate this idea in our setting, given a properly edge-coloured $K_{n,n}$, we fix $q\in (0,1)$ and select every edge of $K_{n,n}$ with probability $q/n$. Then we delete all edges which share vertices or colours with other selected edges. This will certainly produce a rainbow matching. The matching produced like this is often called a ``bite''. How big will it be? Unfortunately not very big. The expected number of edges in the bite will be $qn(1-q/n)^{3(n-1)}$ which is roughly $qne^{-3q}$ for large $n$. So the maximum size of the matching would be $n/3e$ achieved by $q=1/3$. R\"odl's brilliant idea was to perform several small bites one after another, deleting the vertices/colours used in each bite from the rest of the graph. Although after the first bite, the remaining graph will no longer be complete, it still turns out to be possible to repeatedly bite until the remainder has size $<O(n^{1-\alpha})$. This is based on the phenomenon that edges/vertices not used on each bite have pseudorandomness properties. 

Our key new idea is to show that this matching has nice ``expansion'' properties. In fact, we only need to analyse these properties for the first bite. The structure of our proof is the following:
\begin{enumerate}
\item [(S1)] Obtain $M_0$ rainbow matching via the first bite and show it satisfies certain expansion properties.
\item [(S2)] Delete vertices and colours of $M_0$ from $K_{n,n}$. The remaining graph will still have pseudorandom properties with respect to both colours and vertices, therefore we can extend $M_0$ to  a larger rainbow matching $M$ of size $n-n^{1-\alpha}$. This step is done via using R\"odl's nibble as a black box on coloured pseudorandom graphs.
\item [(S3)] The expansion properties that $M_0$ had can be transferred to $M$ which will allow us to do switching-type arguments to increase $M$ as long as we have $\log n/\log\log n$ unused colours outside of $M$. We do this iteratively, starting from $M$ and obtaining a new matching of size  one bigger at every step.
\item [(S4)] After at most $O(n^{1-\alpha})$ times we get a matching with remainder at most $O(\log n/\log\log n)$. 
\end{enumerate}

The major part of this paper is devoted to establishing (S1) in Section~\ref{sec:expansion}. Next we discuss the notion of expansion we study.  Our approach is heavily inspired by the idea that a ``randomly chosen matching will satisfy pseudorandomness properties''. The pseudorandomness property that we use  is very different from the ones previously used in nibble-type proofs.  It can be summarised as ``the union of a random matching together with an arbitrary nearly regular graph $D$ will have strong expansion properties''. Here is a simplified version of what we prove:
\begin{lemma}Let $0<q\ll 1$,  $q^{-1}\ll d \leq \sqrt n$. Given  $K_{n,n}$  properly edge-coloured by $n$ colours, let $H$ be its subgraph formed by choosing every edge independently with probability $q/n$. Delete all edges of $H$ which share vertices or colours with other edges of $H$ and let $M_0$ be the resulting rainbow matching.  Then
with high probability
\begin{itemize}
\item[(E1)]  every collection $D$ of $d$ colours in $K_{n,n}$, and every set $S$ of $n/q^4d$ vertices 
there are at least $(1-q)n$ vertices that can be reached from $S$ by a $D$-$M_0$ alternating path of length three, i.e., a path whose first and last edge is in $D$ and the middle edge is in $M_0$.
\end{itemize}
\end{lemma}
Notice that (E1) only provides expansion  for large sets $S$ but for our purposes we need it  to hold for all sets. After we  extend the matching $M_0$ to a larger rainbow matching $M$  of size roughly $n-n^{1-\alpha}$ as desrcibed in (S2) we are able to iterate (E1) if we restrict to larger collections of colours and longer paths. In particular, we obtain the following refinement of (E1) with respect to $M$.

\begin{itemize}
    \item [(E2)] \textit{For $d=\log{n}/\log{\log{n}}$ and any collection $D$ of $d$ colours there exists a set  of vertices $V_0$  of size at most $qn$ such that the following holds. Every vertex not in $V_0$ can reach  all but $qn$ vertices  via $D-M$ alternating  rainbow paths of length $O(\log{n}/\log{\log{n}})$. }
\end{itemize}

Note that (E2)  also implies that between any two vertices of $K_{n,n}$ lying in different sides of the bipartition there is a $D-M$ alternating  rainbow path of length $O(\log{n}/\log{\log{n}})$. This  property is enough to  perform the modifications described in (S3). We find an alternating rainbow path to extend the matching $M$ by one edge at a time much like in standard proofs of say Hall's Matching Theorem. In the applications of (E2) we let $D$ to be the set of unused colours on $M$. The condition $|D|=\log n/\log\log n$ allows us to do iterations and also tells us when the process must stop and the enlargement of the matching is no longer possible. The reason why we need this many unused colours is because the length of the alternating paths, used to perform the switching-type arguments, can be as large as $\log n/\log\log n$, and we need to guarantee that these paths are rainbow. So we can repeat (S3) until the number of unused colours on $M$ is $O(\log{n}/\log{\log{n}})$.

\section{Preliminaries}
We will use asymptotic ``$\ll$'' notation to state our intermediate lemmas. When we write ``$\delta\ll  \eps$'' in the statement of a result, it means ``for all $\eps>0$ and sufficiently small $\delta>0$, the following statement is true''. In particular ``$n^{-1}\ll \eps$'' means ``for all $\eps>0$ and sufficiently large $n$, the following is true''. When we chain several inequalities like this, the quantity on the left is small relative to all constants on the right. For example ``$n^{-1}\ll \delta \ll \eps$'' means ``for all $\eps>0$, there is a $\delta_0$ such that for positive $\delta<\delta_0$ and sufficiently large $n$, the following is true''. 
Sometimes we will abuse this notation and write ``$n\gg \eps^{-1}$'' to mean ``$n^{-1}\ll \eps$''.\Alexey{[Alexey: is this last sentence okay?]}

For any positive reals $a,b\in \mathbb R$, we use ``$x=a\pm b$'' to mean ``$a-b\leq x\leq  a+b$''. We also use the same notation with more than one instance of ``$\pm$''. We will write expressions of the form ``$f=g$'' where $f$ and $g$ are functions involving, one or more instances of ``$\pm$''.
To interpret such an expression, first define $\max_{\pm} f$ to be the maximum value of $f$ taken over all possible assignments of $+/-$  to each ``$\pm$'' symbol. Similarly define  $\min_{\pm} f$. Then we say that ``$f=g$'' if $\max_{\pm}f\leq \max_{\pm}g$ and $\min_{\pm}f\geq \min_{\pm}g$ are both true.

For a graph $G$, the set of edges of $G$ is denoted by $E(G)$ and the set of vertices of $G$ is denoted by $V(G)$. The set of neighbours of $v$ is denoted by $N_G(v)$, and  $d_G(v)=|N_G(v)|$.
For a coloured graph $G$ and a colour $c$, denote by $E_G(c)$ the set of edges of colour $c$ in $G$, and denote $V_G(c)$ for the set of vertices touching colour $c$ edges. In all of these, we omit the ``$G$'' subscript when the graph $G$ is clear from context. For a properly edge-coloured graph $G$, $C\subseteq C(G)$  and $v\in V(G)$ we denote by $N_C(v)$ the set of vertices $w$ such that $vw\in E(G)$ and $c(vw)\in C$.
For a graph $G$, a set of vertices $A$, let $G[A]$ denote the induced subgraph of $G$ on $A$. For a  coloured graph $G$, a set of colours $C\subseteq C(G)$,  we let $G[C]$ to be the subgraph of $G$ induced by edges of colours in $C$. 

Let  $G,H$ be   graphs on the same vertex set $V$.  We say that a path $x_1x_2\dots x_t$ is $G$-$H$ alternating if, for for odd $i$, $x_ix_{i+1}\in E(G)$ and for even $i$, $x_ix_{i+1}\in E(H)$ (or in other words, the first edge is in $G$, and thereafter the edges alternate between $G$ and $H$).
For a set $S\subseteq V(G\cup H)$, we use $N^t_{G,H}(S)$ to denote the set of vertices $v$ to which there is a  $G$-$H$ alternating path of length $t$ from some $s\in S$.

\subsection{Probabilistic tools}
Here we gather basic probabilistic tools that we use. We use the Chernoff bounds. Most of these can be found in textbooks on the probabilistic method such as~\cite{molloy2013graph}.
\begin{lemma}[Chernoff bounds, \cite{molloy2013graph}]Given a binomially distributed variable $X\in Bin(n, p)$  for all  $0<a\leq 3/2$ we have
 $$\Prob{[|X-\E[X]|\geq a \E[X]]}\leq 2e^{-\frac{a^2}{3}\E[X]}$$.
\end{lemma}
Given a product space $\Omega=\prod_{i=1}^n\Omega_i$ and a random variable $X:\Omega\to \mathbb{R}$ 
we  make the following definitions.
\begin{itemize}
\item  Suppose that there is a constant $c$ such that changing $\omega\in \Omega$ in any one coordinate changes $X(\omega)$ by at most $c$. Then we say that $X$ is $c$-Lipschitz.
\item Suppose that for any $s\in \mathbb{N}$ and $\omega\in \Omega$ with $X(\omega)\geq s$ there is a set $I\subseteq \{1, \dots, n\}$ with $|I|\leq rs$ such that every $\omega'$ which agrees with $\omega$ on coordinates in $I$ also has $X(\omega')\geq s$. Then we say that $X$ is $r$-certifiable.
\end{itemize}

We'll use the following two versions of Azuma's inequality.
\begin{lemma}[Azuma's Inequality, \cite{molloy2013graph}]\label{Lemma_Azuma}
For a product space $\Omega=\prod_{i=1}^n\Omega_i$ and a $c$-Lipschitz random variable $X:\Omega\to \mathbb{R}$, we have
$$\Prob\left(|X-\E(X)|>t \right)\leq 2e^{\frac{-t^2}{nc^2}}$$
\end{lemma}

\begin{lemma}[Azuma's Inequality for 0/1 product spaces, \cite{alon1997nearly, kim2002asymmetry}]\label{Lemma_Azuma_variance}
Let $\Omega=\{0,1\}^n$ with the $i$th coordinate of an element of $\Omega$ equal to 1 with probability $p_i$. Let $X$ be a $c$-Lipschitz random variable on $\Omega$. Set $\sigma^2=c^2\sum_{i=1}^n p_i(1-p_i)$. For all  $t\leq 2\sigma/c$, we have
$$\Prob\left(|X-\E(X)|>t\sigma \right)\leq 2e^{\frac{-t^2}{4}}$$
\end{lemma}
We'll also use the following version of Talagrand's inequality.
\begin{lemma}[Talagrand Inequality, \cite{molloy2013graph}]\label{Lemma_Talagrand}
For a product space $\Omega=\prod_{i=1}^n\Omega_i$ and a $c$-Lipschitz, $r$-certifiable random variable $X:\Omega\to \mathbb{R}$, we have
$$\Prob\left(|X-\E(X)|>t+60c\sqrt{r\E(X)} \right)\leq 4e^{\frac{-t^2}{8c^2r\E(X)}}$$
\end{lemma}

We call a  bipartite graph $G$ with parts $X,Y$ is $(\eps,p,n)$-\emph{regular} if

\begin{itemize}\item [(P1)] $|X|=|Y|=n(1\pm n^{-\eps})$,
\item [(P2)]  $d(v)=pn(1\pm n^{-\eps})$
\end{itemize}

Furthermore, $G$ is $(\eps,p,n)$-\emph{typical} if 
\begin{itemize}\item [(P3)] for every $u,v\in X$ or $u,v\in Y$ we have
 $|N(u)\cap N(v)|=p^2n(1\pm n^{-\eps})$. 
 \end{itemize}
A bipartite graph $G$ with bipartition $(X,Y)$ and colour set $C$ is called  \emph{coloured $(\eps,p,n)$-regular}/ \emph{coloured $(\eps,p,n)$-typical} if it is properly edge-coloured and the following hold:
\begin{itemize}\item [(P4)] $G$ is (uncoloured) $(\eps,p,n)$-regular/$(\eps,p,n)$-typical.
\item [(P5)] Define  $G_{X,C}$ to be the bipartite graph with  vertex bipartition $(X,C)$ where  $xc$ is an edge for $x\in X, c\in C$ if there exists some $y\in Y$ such that $xy\in E(G)$ and $c(xy)=c$. Define $G_{Y,C}$  analogously. We require  both $G_{X,C}$ and  $G_{Y,C}$ to be $(\eps, p,n)$-regular/$(\eps, p,n)$-typical.
\end{itemize}
Note that a coloured $(\eps,p,n)$-regular graph $G$ is \emph{coloured} $(\eps,p,n)$-regular if additionally $|C(G)|=(1\pm n^{-\eps})n$ and every colour $c\in C(G)$ has  $|E_G(c)|=(1\pm n^{-\eps})pn$. 
Similarly, a  properly edge-coloured $(\eps,p,n)$-typical graph is \emph{coloured} $(\eps,p,n)$-typical if these happen and additionally every pair of colours $c,c'$ have $|V_G(c)\cap V_G(c')\cap X|=(1\pm n^{-\eps})p^2n$ and $|V_G(c)\cap V_G(c')\cap Y|=(1\pm n^{-\eps})p^2n$. 
 
Frankl and R\"odl \cite{frankl1985near} (also Pippenger, unpublished) showed that every $n$-vertex hypergraph with $(1\pm \epsilon)pn$ degrees and codegrees at most one has a matching of order $(1-\gamma)n$. 
A corollary of this is that every coloured $(\gamma, p, n)$-regular graph has a rainbow matching of order $(1-\gamma)n$ (to see this, associate a hypergraph with the  $(\gamma, \delta, n)$-regular graph as explained in the introduction and apply their theorem). 
We'll need the following standard version (which appeared in the literature before) of this result where the error term $\gamma n$ is polynomially related with $n$.  

\begin{lemma}\label{Lemma_variant_of_MPS_nearly_perfect_matching}
Let $n^{-1}\ll \gamma\ll\eps$ and $n^{-1}\ll p\leq 1$.
Every coloured $(\eps, p, n)$-regular bipartite graph  $G$ has  a  rainbow matching of size $n-n^{1-\gamma}$.
\end{lemma}
\begin{proof}
Notice that $G$ is balanced bipartite with parts of size $(1\pm n^{-\eps})n$,  every vertex has degree $(1\pm n^{-\eps})pn$, and every colour occurs at most $(1+ n^{-\eps})pn$ times. Now the lemma is strictly weaker than Lemma~4.6 from \cite{montgomery2019decompositions} (applied with $n=n, \gamma=n^{-\eps}, \delta=p, p=n^{-\gamma}, \ell=1$). 
\end{proof}

The following lemma shows that a random subgraph of a typical bipartite graph is typical. There are two notions of what ``random subgraph'' means here. The most important one is to consider the subgraph formed by deleting every vertex/colour independently with fixed probability (case (a) below). We use the second case in Section 6 to reduce the problem of finding large matchings in Steiner systems to finding large matchings in special typical graphs.
\begin{lemma}\label{Lemma_random_subgraph_typical_graph}
Let $n^{-1}\ll p,q, \eps \leq 1$. Let $G$ be a coloured $(\eps, p, n)$-typical bipartite graph with bipartition $X,Y$ and colour set $C$. Let $X'\subseteq X, Y'\subseteq Y, C'\subseteq C$ be random sets obtained as follows:
\begin{enumerate}[(a)]
    \item Every vertex/colour ends up in $X',Y',C'$ independently with probability $q$.
    \item Suppose we can label $X=\{x_1, \dots, x_n\}$, $Y=\{y_1, \dots, y_n\}$, $C=\{c_1, \dots, c_n\}$ such that if $x_iy_j$ is an edge of $G$ of colour $c_k$ then all $i, j, k$ must be distinct. 
    Form $X',Y',C'$ by choosing disjoint set of indices $I_X, I_Y, I_C\subseteq [n]$ such that independently every $i\in [n]$ is placed in $I_X, I_Y$ and $I_C$ with probability $q$ and in none of them with probability $1-3q$. Set $X'=\{x_i:i\in I_X\}$, $Y'=\{y_i:i\in I_Y\}$, $C'=\{c_i:i\in I_C\}$
\end{enumerate}
Let $H$ formed by colour $C'$ edges going from $X'$ to $Y'$. Then with  probability at least $1-e^{-n^{1-\eps/2}}$, $H$ is coloured $(\eps/8, qp,qn)$-typical.
\end{lemma}
\begin{proof}
We will show that with probability  $1-\frac13e^{-n^{1-\eps/2}}$, $H$ is (uncoloured)  $(\eps/8, qp,qn)$-typical. By symmetry between $X,Y,C$, the same proof shows that $H_{X,C}$ and $H_{Y,C}$ are $(\eps/8, qp,qn)$-typical. Thus we will have that with probability $1-e^{-n^{1-\eps/2}}$ all  of $H$, $H_{X,C}$, and $H_{Y,C}$ are $(\eps/8, qp,qn)$-typical, or equivalently $H$ is coloured $(\eps/8, qp,qn)$-typical. To give a unified proof of both statements $(a), (b)$ we will use Azuma's inequality.

Let $u,v$ be two vertices on the same side of $G$, and $y\in N_G(u)\cap N_G(v)$. Without loss of generality, we may suppose that $u,v\in X, y\in Y$. 
Notice that in both (a) and (b) we have $\Prob(u\in X') =q$,  $\Prob(c(uy)\in C', y\in Y')= q^2$, and $\Prob(c(uy), c(vy)\in C', y\in Y')=q^3$. Indeed, here we use that $c(uy)\in C'$, $c(vy)\in C'$, and $y\in Y'$ are independent events which is true in  case (a) trivially. For case (b), suppose $y=y_i$, $c(uy)=c_j$, $c(vy)=c_k$, it is enough to show that all three indices $i,j,k$ are distinct. Indeed, since $uy, vy\in E(G)$, it follows that $i$ and $k$, and $j$ and $k$ are distinct. Finally, $j$ and $k$ are distinct since the edge-colouring is proper. Since $G$ is    $(\eps, p, n)$-typical,  we have $|X|,|Y| =(1\pm n^{-\eps})n$,
$|N_G(u)|= (1\pm n^{-\eps})pn$ and $|N_G(u)\cap N_G(v)|=(1\pm n^{-\eps})p^2n$. 
Thus we have
\begin{align*}
 \E[|X'|],\E[|Y'|]&=(1\pm n^{-\eps})qn\\
 \E[|N_{C'}(u)\cap Y'|]&=(1\pm n^{-\eps})q^2pn,\textit{ for all }   u\in X \\ 
\E[|N_{C'}(u)\cap N_{C'}(v)\cap Y'|]&= (1\pm n^{-\eps})q^3p^2n, \textit{ for all }  u, v\in X.
\end{align*}
Notice that these random variables are all $2$-Lipschitz and are each affected by at most $3(1+n^{-\eps})n \leq 4n$ coordinates. By Azuma's inequality we get that for $t=q^3p^2n^{1-\eps/8}/2$ with probability  $1-2e^{-t^2/16n} \geq 1-e^{-n^{1-3\eps/8}}$,  each one of them are within $t$ of their expectations. By taking union  bound over all vertices and colours we obtain that  with probability at least $ 1-(n+n^{-\eps})^{3}e^{-n^{1-3\eps/8}} \geq 1-\frac13e^{-n^{1-\eps/2}}$ they are all simultaneously within $t$ of their expectations. So now the result follows from the definition of $(\eps/8, qp,qn)$-typicality and the fact that 
$t+n^{1-\eps}<q^3p^2n^{1-\eps/8}$. 
\end{proof}

 We will need the following result about typical graphs. It is a bipartite variation of Lemma 5.5 from~\cite{montgomery2019decompositions} (see also Lemma 2.1 in \cite{AKS}),   whose proof is straightforward from the original version.   
\begin{lemma}\label{lem:typical-jumbled} Let $n\in \mathbb{N}$, $\eps,p, \gamma\in (0,1]$ with $8n^{-\eps}\leq \gamma$. Then every  $(\eps, p,n)$-typical bipartite graph $H$ with sides $X,Y$ satisfies the following. For every pair of subsets $A\subseteq X$, $B\subseteq Y$ with 
$|B|\geq \gamma^{-1}p^{-2}$:
$$|e(A,B)- p|A||B||\leq 2|A|^{\frac12}|B|\gamma^{\frac12}n^{\frac12}p.$$
\end{lemma}

\begin{proof}
Let $\mathrm{Adj}_H$ be the adjacency matrix of $H$, and let $M=\mathrm{Adj}_H-pJ$ where $J$ is the appropriately-sized all-ones matrix.
Notice that for every pair of distinct vertices $y,y'\in Y$,  we have
\begin{eqnarray}
\sum_{v\in X}M_{y,v}M_{y',v}&=& d_H(y,y')-p(d(y)+d(y'))+p^2|X|\leq (1+n^{-\eps})p^2 n- 2(1-n^{-\eps})p^2 n+p^2(1+n^{-\eps})n \nonumber \\
&\leq&\gamma p^2 n/2.\label{exponew}
\end{eqnarray}
Next notice that we have
\begin{eqnarray*}
\big|e(A,B)- p|A||B|\big|^2 &=&\left(\sum_{x\in A}\sum_{y\in B} M_{x,y}\right)^2
 \leq |A|\sum_{x\in A}\left( \sum_{y\in B} M_{x,y}\right)^2
 \leq |A|\sum_{x\in X}\left( \sum_{y\in B} M_{x,y}\right)^2\\
 &=& |A|\sum_{x\in X}\left( \sum_{y\in B} M_{x,y}^2\right) +|A|\sum_{x\in X}\left(\sum_{y\neq y'\in B} M_{x,y}M_{x,y'}\right)\\
 &\leq& |X||A||B|+ |A| \sum_{y\neq y'\in B}\left(\sum_{x\in X} M_{x,y}M_{x,y'}\right)\\
 &\overset{\eqref{exponew}}{\leq}&  (1+n^{-\eps})n|A||B|+ |A|\sum_{y\neq y'\in B} \frac{\gamma  p^2n}{2}
 \leq  (1+n^{-\eps})n|A||B|+ |A||B|^2 \gamma  p^2n/2\\
 &\leq&  2|A||B|^2 \gamma  p^2n
\end{eqnarray*}
Here the first inequality comes from the Cauchy-Schwarz inequality
and the last inequality comes from $|B|\geq \gamma^{-1}p^{-2}$. Taking square roots gives the result.
\end{proof}

Next we show that the above result implies that for a coloured typical graph $G$ and any  set of $d$ many colours in $G$, the subgraph of $G$ induced by the edges of colours in $D$ can have at most  $O(n/d)$ many vertices of small degree, for $d=O(n^{\eps}).$

\begin{lemma}\label{lem:fewisolated}Let $n^{-1} \ll p, \eps \leq 1$, $16\leq 8p^2d\leq n^{\eps}$. Suppose $G$ is a coloured $(\eps,p,n)$-typical bipartite graph with bipartition $(X,Y)$ and colour set $C$. Then for any set of $d$ colours $D$  the subgraph $G[D]\subseteq G$ induced by edges of colours in $D$ has at most  $\leq 32p^{-2}n/d$ many vertices of degree less than $pd/2$.
\end{lemma}
\begin{proof} Let $J$ be the set of vertices in $G[D]$ of degree less than $pd/2$. We will show that $|J\cap X| \leq 16p^{-2}n/d$. Similarly one can prove that  $|J\cap Y| \leq 16p^{-2}n/d$. Recall $G_{X,C}$ is defined on the vertex bipartition  $(X,C)$ where we put an edge $xc$ if there is some $y\in Y$ such that $xy\in E(G)$ and $c(xy)=c$. By definition of coloured typical graphs, $G_{X,C}$ is $(\eps,p,n)$-typical hence  we can apply Lemma~\ref{lem:typical-jumbled} to the set $J\cap X\subseteq X$ and $D\subseteq C$ with $\gamma =p^{-2} d^{-1}$ (so that $|D|\geq \gamma^{-1}p^{-2}$). We obtain 
$$p|J\cap X|d - 2p|J\cap X|^{1/2}d\gamma^{1/2}n^{1/2}\leq e_{G_{X,C}}(J\cap X, D) <  |J\cap X| pd/2.$$
From here it follows that $|J\cap X|\leq 16 \gamma n\leq 16 p^{-2}n/d.$
\end{proof}

\section{Expansion and its properties} 
\label{sec:expansion}
The proof of our main technical theorem, which we present in this section, is  based on finding a nearly spanning randomized rainbow matching $M$ which ``expands'' in some sense. We then show that these expansion properties can be used to alter $M$ via series of switchings along alternating paths to obtain a new matching  covering all but $O(\log n/\log\log n)$ vertices.

\subsection{Typical coloured graphs are expanding}
 In this subsection we prove that every typical graph has a large matching which is  ``expanding''  with respect to any small collection of colours. First we define what we mean by expanding. 
 Since by itself a matching is clearly not expanding graph, we will always speak about expansion properties of a \emph{union of two graphs}, one of which will always be a matching.
 Let  $G,H$ be two graphs.  Recall that for a set $S\subseteq V(G)\cup V(H)$,  we use $N^t_{G,H}(S)$ to denote the set of vertices $v$ to which there is a length $t$ path  from some $s\in S$ whose edges alternate between $G$ and $H$ with the first edge belonging to $G$. The following definition is key in this paper.
\begin{definition}[Expander]\label{def:expander} For a  matching $M$ and a bipartite graph $D$ we say that $(D, M)$ is a $(d, A, \eps,n)$-expander if   every vertex set $S\subseteq X$ or $S\subseteq Y$ with $|S|\geq An/d$  has a subset $S'$ with $|S'|=An/d^2$ and $|N^4_{D, M}(S')|\geq (1-\eps)n$, where $(X,Y)$ is the bipartition of $D\cup M$ with $|X|,|Y|\geq An/d$. \end{definition}

Note that, since in this definition the last edge on the $4$-edge path is from $M$, it follows that if $(D,M)$ is a $(d, A, \eps,n)$-expander then $|M|\geq (1-\eps)n$. We also want to point out few additional subtleties. First, it would be more natural to ask $|N^2_{D, M}(S)|$ to be of order $(1-o(1))n$. However, it is not true that the second neighbourhoods expand (see details in the proof sketch of Lemma~\ref{Lemma_main_expansion}). Second, we have a stronger requirement that $S$ has a subset of size $\Theta(n/d^2)$ which expands. This is done for the following two technical purposes.

We are able to show that every coloured pseudorandom graph $G$ has a random rainbow matching such for any subgraph $D$ induced by edges of any collection of $d=\log{n}/\log{\log{n}}$ many colours, every $S$ of size roughly $n/d$ expands (in the above sense) with probability  at least $1-e^{-|S|}$ (see Lemma~\ref{Lemma_Mnp_second_neighbourhood_expansion}). Then we would like to claim, by taking the union bound, that with high probability all sets $S$ expand simultaneously. Unfortunately, the probability that there is some $S$ which does not expand is at most roughly $\binom{n}{|S|} e^{-|S|} \gg 1$. Instead, for each non-expanding set $S$ we find a smaller set $S'$ of size $|S|/d$ which 'captures' the expansion properties of $S$.  Now the union bound gives us that the probability that some set $S$ does not expand is at most $\binom{n}{|S|/d} e^{-|S|} \ll 1$. (This idea is similar to the applications of containers widely used in studying $H$-free graphs for fixed $H$, where one shows that there is a  collection of containers of bounded size which contain all $H$-free graphs of certain size). 

The second reason to have a smaller subset $S'$ which captures the expansion of $S$ is for finding rainbow $D-M$ alternating paths between almost all $x\in X$ and $y\in Y$ (see Section~\ref{sec:switching} for details). Let's assume $S\subseteq Y$. As the first step to achieve this, we need to show  that almost all vertices $y\in Y$ have some rainbow alternating $D-M$ path of length four starting at some $s\in S$ and ending at $y$. Furthermore, we need that each of these paths avoids a prescribed set of colours and vertices of order $\eps d$. Since $|N^4_{D,M}(S')|=(1-o(1))n$ for each $y\in Y$ there is a $D-M$-alternating path that starts at some $s\in S'$ and ends at $y$.  Some of these paths can be \emph{bad} if either they are not rainbow or they do not avoid the prescribed set of vertices and colours. The number of such bad paths is at most roughly $\eps d^2|S'|$ where $|S'|$ factor comes for the choice of starting vertex in $S'$, $\eps d$ comes from using a forbidden vertex or a colour and the second $d$ factor is due to the fact that  $\Delta(D)\leq d$ (see Lemma~\ref{Lemma_expansion_conversion}). So using that $|S'|\approx n/d^2$ we conclude that the number of bad paths is at most $\eps n$.

The main result of this section is to prove the following expansion properties of coloured typical bipartite graphs.

\begin{lemma} [Main expansion lemma]\label{Lemma_main_expansion}
Let $n^{-1}\ll q \ll p \leq 1$, $n^{-1}\ll \gamma \ll \eps \ll 1$ and $n^{-\eps/2} \leq d^{-1}\ll q$.  Suppose $G$ is a coloured $(\eps, p, n)$-typical bipartite graph.  Then there is a randomized rainbow matching $M$ in $G$ with the following property.  

For any  bipartite graph $D$ on the same biparition as $G$ with $\Delta(D)\leq d$ and  at most $96p^{-2}n/d$ vertices of degree less than $pd/6$, with probability at least $1-2e^{-n^{1-\eps}}$  the following hold:
\begin{enumerate}[(i)]
\item $|M|\geq (1-n^{-\gamma})n$,
\item $(D, M)$ is a $(d, q^{-4}, q, n)$-expander.
\end{enumerate}
\end{lemma}

The proof of Lemma~\ref{Lemma_main_expansion} is technical. Here we present a quick sketch. The randomized matching $M$ in Lemma~\ref{Lemma_main_expansion} will be composed of two bits --- $M_0$ and  $M_1$. 
 First we choose $M_0$ by picking every edge in $G$ with probability $q/n$ and deleting all colour and vertex collisions, for some $0<q\ll 1$. This matching will be of size roughly $q n$ and will satisfy certain ``expansion'' properties in second and third neighbourhoods which we describe next. Fix some $D$ as in the statement of Lemma~\ref{Lemma_main_expansion}  and suppose $(X,Y)$ is the bipartition of $M_0\cup D$. We first show that $M_0$ has the following property:
 
 \begin{itemize}
     \item  \textbf{Second neighbourhood expansion:}  For each set $S\subseteq X $ or  $S\subseteq Y$ of size roughly $n/d$ we have $|N_{M_0,D}^2(S)|= (1-o(1))n$  with probability $1-e^{-O(|S|)}$(see Lemma~\ref{Lemma_Mnp_second_neighbourhood_expansion}). 
 \end{itemize}

 The above property simply means for say $S\subseteq X$ that if we follow edges coming out of $S$  that belong to $M_0$ and then follow the edges of $D$ we reach almost all of $X$. Notice that we cannot prove  that the second neighbourhood is large for all sets $S$ simultaneously. Indeed, let $D$ be a disjoint union of complete bipartite graphs of size $d$ on $X,Y$ with $|X|=|Y|=n$. Now let $M_0$ be any perfect matching on $K_{n,n}$. Now let $H$ any union of disjoint $K_{d,d}$'s in $D$.  If we let $S=V(H)\cap X$, then  $|N^2_{D,M}(S)|=|N_M(V(H)\cap Y)|=|S|$ that is $S$ won't expand.  Note that    it doesn't matter whether we follow the edges in the order of $M$ and then $D$ or otherwise. Indeed, take $S'=N_{M}(S)$, then $|N^2_{M,D}(S')|=|N_D(S)|=|S|=|S'|$. However, if we look at the same example in the third neighbourhood, that is $N^3_{D,M}(S)$ and let's assume $|S|=n/d$ then if $M$ was a raindomly picked perfect matching, it is likely that  $N^2_{D,M}(S)$ will hit a vertex from each $K_{d,d}$ outside of $H$ thus resulting $N^3_{D,M}(S)$ being almost of of $X$. And this is what we prove; using the second neighbourhood expansion and that $M_0$ is picked randomly, we show that  with high probability  \emph{all}  large sets $S$ have expansion in their third neighbourhood.
 \begin{itemize}
     \item \textbf{Third neighbourhood expansion:}  With high probability, all   sets $S\subseteq X $ or  $S\subseteq Y $ of size roughly $n/d$  will have subsets $S'$ of size roughly $n/d^2$ and $|N_{D,M_0}^3(S')|= (1-o(1))n$. (see Lemma~\ref{Lemma_expansion_3rd_neighbourhood}). 
 \end{itemize}
  Finally notice that to obtain the expansion  in the fourth neighbourhood in the sense of Definition~\ref{def:expander}, $M_0$ is not enough, as it is only of size roughly $q n$. That is why we need to extend $M_0$ to a nearly spanning rainbow matching.
 Let  $H$ be obtained from $G$ by deleting vertices colours of $M_0$. Since $G$ was coloured regular and $M_0$ was picked randomly, $H$ will be coloured regular as well (Lemma~\ref{Lemma_degree_concentration}). We find a nearly spanning rainbow matching $M_1$ in $H$, which by definition of $H$, will be edge and colour disjoint from $M_0$. This is done by applying Lemma~\ref{Lemma_variant_of_MPS_nearly_perfect_matching} to $H$, which gives a rainbow matching $M_1$ of size roughly $n-|M_0|- n^{1-\gamma}$. Then, taking  $M=M_0\cup M_1$, for any nearly regular graph $D$ on $V(G)$ with high probability we obtain that all large sets will expand as in the Definition~\ref{def:expander}.

We start with an easy lemma exhibiting a feature of nearly-regular graphs. 

\begin{lemma}\label{Lemma_regular_bipartite_containers}
For  $\kappa\leq 1 \leq d$, let $D$ be a bipartite graph with bipartition $(X,Y)$ and $\Delta(D)\leq d$. For any $S\subseteq X$ or $S\subseteq Y$ with  $|S|\geq 2d$ and every $s\in S$ satisfying $d_D(s)\geq \kappa d$, there is a set $S'\subseteq S$ such that $|S'|\leq |S|/d$ and  $|N_D(S')|\geq \kappa |S|/4$.
\end{lemma}

\begin{proof} Take the maximal collection of vertex-disjoint stars of size $\kappa d/2$ in $D$ whose centers are in $S$ and let $F$ be the vertex set of their union. We are done if $| F\cap X|\geq |S|/2d$. Indeed, in this case any set $S'\subseteq S$   containing $|S|/2d$ many of the centers of the stars of $F$ satisfies the lemma. So, we may assume  $|F\cap X|< |S|/2d$. Since $d_D(x)\geq \kappa d$, by maximality of $F$, we have  $|N_D(x)\cap F|\geq \kappa d/2$ for all $x\in S\setminus F$. On the other hand,  since $\Delta(D)\leq d$, for any $y\in Y\cap F$, $|N_D(y)|\leq d$.  Thus, 

$$ \frac{\kappa d}2|S\setminus F| \leq e(S \setminus F, Y\cap F)\leq d |Y\cap F|.$$

This implies $$|Y\cap F|\geq \frac{\kappa}2{|S\setminus F|}=\frac{\kappa}2(|S|-|X\cap F| ) >  \frac{\kappa|S|}{4},$$
where in the last inequality we used $|F\cap X|<|S|/2d$ and $d\geq 1$. Now $S'=F\cap X$ has $|S'|<|S|/2d$ and $|N_D(S')|\geq|Y\cap F|>\kappa |S|/4$ as required. 
\end{proof}

\begin{lemma}\label{Lemma_Gnp_second_neighbourhood_expansion}
Let $n^{-1}\ll q \ll p \leq 1$, $n^{-1}\ll \eps <1$ and $q^{-1}\ll d\leq qn/8$. \Alexey{[Added this upper bound on d for applying Lemma 3.3 ]}
Let $G$, $D$ be two bipartite graphs on the same vertex set with bipartition $(X,Y)$ such that $G$ is coloured $(\eps,p,n)$-typical, $\Delta(D)\leq d$ and all but at most $96p^{-2}n/d$ vertices have degrees less than $pd/6$ in $D$. Let $H$ be derived from $G$ by picking every edge independently with probability $q^2/n$. Let  $S\subseteq X$  or $S\subseteq Y$ with $|S|= \frac{n}{d q^3}$. Then with probability at least $1-e^{- q^7 |S|}$ we have  $|N^2_{H, D}(S)|\geq (1-q)n$.
\end{lemma}
\begin{proof}
Let $J$ be the set of vertices $v\in X\cup Y$ such that $d_{D}(v)< pd/6$. Without loss of generality, let us assume $S\subseteq X$. Denote $B=\{x\in X\setminus J: e_G(N_D(x),S)\leq p^2d|S|/30\}$. 

\begin{claim}$|B|\leq  qn/4$.  
\end{claim}
\begin{proof}
Suppose, for contradiction, that $|B|> qn/4$. By Lemma~\ref{Lemma_regular_bipartite_containers} there is $B'\subseteq B$ with  $|B'|\leq |B|/d$ such that $|N_D(B')|\geq p|B|/24 > pqn/96$. We apply Lemma~\ref{lem:typical-jumbled} with $\gamma_{\ref{lem:typical-jumbled}} =p q/9600$ and obtain

 $$e_G(N_D(B'),S)\geq p|N_D(B')||S|-2p|N_D(B')|^{\frac12}|S|(\gamma n)^{\frac12}> \frac{4}{5}p|N_D(B')||S|\geq \frac{1}{30}p^2|B||S|.$$

On the other hand,  by the definition of $B'$ we have $e_G(N_D(B'),S) \leq  p^2d|S|/30 \cdot |B'| \leq p^2\frac{|B||S|}{30},$
which is a contradiction.
\end{proof}

\begin{claim}
For every $x\in X\setminus (B\cup J)$,  $\Prob[x\not\in N^2_{H, D}(S)]\leq q/4.$
\end{claim}
\begin{proof}
For each such $x$, define the set $S_x =\{s\in S: |N_G(s)\cap N_D(x)|\geq p^2d/60\}$. Using $x\notin B$ we get

\begin{align*}\frac{ p^2d|S|}{30}<e_G(N_D(x),S)\leq |S_x||N_D(x)| + |S\setminus S_x|\frac{p^2d}{{60}}  \leq d|S_x|  + \frac{p^2d|S|}{60},
\end{align*}
implying that $|S_x|\geq  p^2 |S|/60$.

 Now we  compute the probability of the event $x\notin N^2_{H,D}(S)$.
\begin{align*}
\Prob[x\notin N^2_{H, D}(S)]&=\Prob[\forall s\in S, y\in N_D(x)\cap N_G(s) \textit { we have } sy\not\in E(H)]\\
&= \prod_{s\in S}\prod_{y\in N_D(x)\cap N_G(s)} \Prob[sy\not\in E(H)] 
= \prod_{s\in S} \left(1-\frac{q^2}{n}\right)^{|N_D(x)\cap N_G(s)|} 
\leq \prod_{s\in S_x} \left(1-\frac{q^2}{n}\right)^{|N_D(x)\cap N_G(s)|} \\
&\leq \prod_{s\in S_x} \left(1-\frac{q^2}{n}\right)^{p^2d/60}
\leq \prod_{s\in S_x} e^{- p^2q^2d/60n} \leq (e^{- p^2q^2d/{60}n})^{ p^2 |S|/60} =e^{-p^4q^{-1}/3600}\leq q/4.
\end{align*}
Here the second equation comes from independence of the events ``$sy\in E(H)$'', the first inequality comes from $S_x\subseteq S$, the second one from the definition of $S_x$, the third one comes from $1-x\leq e^{-x}$, the fourth one comes from $|S_x|\geq  p^2 |S|/60$, and the last one holds since $q\ll p \leq 1$.
\end{proof}

By linearity of expectation we have $\E[|N^2_{H,D}(S)|]\geq  (1-q/4)(|X|-|B|-|J|)\geq (1-q/4)(n-n^{1-\eps}-qn/4-96p^{-2}n/d)\geq (1-q/2)n$.   Notice that the random variable $|N^2_{H,D}(S)|$ is defined on the product space $\Omega$ consisting of all the edges in $G$ from $S$ to $Y$, where the probability of every coordinate being one is $q^2/n$. This product space has $e(S,Y)=|S|pn(1\pm n^{-\eps})$ coordinates. \Alexey{[We incorrectly said that there were $|S||Y|$ coordinates before. Check that I implemented this change correctly.]} 
Notice that $|N^2_{H, D}(S)|$ is $d$-Lipchitz, thus we can apply Lemma \ref{Lemma_Azuma_variance}. 
Let $\sigma^2= d^2\sum_{i\in \Omega}{\left(1-\frac{q^2}{n}\right)\frac{q^2}{n}}$
Note that 
$$\frac{qd\sqrt{p|S|}}{2} \leq \sigma =  qd \sqrt{\sum_{i\in \Omega}{\left(1-\frac{q^2}{n}\right)\frac{1}{n}}}\leq 2qd \sqrt{p|S|}.$$

So let $t=q^3\sqrt{|S|}/4$. Note that $td \leq 2\sigma$ and $\sigma t\leq (2qd \sqrt{p|S|})(q^3\sqrt{|S|}/4)\leq qn/2,$ since $|S|=n/dq^3.$
Thus by Azuma's inequality we have:

$$\Prob[|N^2_{H, D}(S)|\leq (1-q)n]\leq \Prob[|N^2_{H, D}(S)|\leq \E[|N^2_{H, D}(S)|]- \sigma t]\leq 2e^{-t^2/4}= 2e^{-q^6|S|/64}\leq e^{-q^7|S|}.$$
\end{proof}

The graph $H$ produced by the previous lemma won't generally be a matching or a rainbow subgraph. The following lemma estimates how many of its edges conflict with other edges due to a vertex or a colour collision. 
\begin{lemma}\label{Lemma_vertex_probabilities}
Let $n^{-1}\ll q \ll p \leq 1$, $n^{-1}\ll \eps <1$. Let $G$ be a bipartite graph with bipartition $(X,Y)$ such that $G$ is  coloured $(\eps,p,n)$-regular. Let $H$ be derived from $G$ by picking every edge independently with probability $q/n$, let $M\subseteq H$ be consisting of edges which don't share any vertices or colours with other edges in $H$, define $H':=H-M$. Then for any set $S\subseteq X$ with $|S|\geq q^{-3}$,
$$\Prob(|N_{H'}(S)|\geq 5q^2|S|)\leq e^{-q^3|S|}.$$
\end{lemma}

\begin{proof} Let $xy\in E(G)$. For $xy$ to be in $M$ we need $xy\in H$ and also $e\not\in H$ for all edges $e$ sharing a vertex or a colour with $xy$. Thus 
\begin{align*}
\Prob[xy\in M]&= \frac qn \left(1-\frac qn \right)^{d_G(x)+d_G(y)+|E_G(c)|-3}\\&=\frac qn\left(1-\frac qn\right)^{3pn(1\pm n^{-\eps})-3}=\frac qn\left(1-\frac qn\right)^{3pn} (1\pm 4q p n^{-\eps}).    
\end{align*}
Here the second equation uses coloured $(\eps,p,n)$-regularity, and the third equation comes from $(1-q/n)^{\pm 3pn^{1-\varepsilon}-3}=(1\pm 4qpn^{-\varepsilon})$.
This gives
\begin{align*}
    \Prob[xy\in E(H')]&=\Prob[xy\in E(H)]-\Prob[xy\in E(M)]\\ &\leq \frac qn\left( 1-\left(1-\frac qn\right)^{3pn} (1- 4q p n^{-\eps}) \right ) \leq \frac{3q^2p}n + 4q^2pn^{-1-\eps}\leq \frac{4q^2}n,
\end{align*}
which implies $\E[|N_{H'}(S)|]\leq 4q^2|S|$. Notice  $|N_{H'}(S)|$ is $3$-Lipschitz since adding or removing an edge $e$ from $H$ can affect at most two neighbouring edges or one edge of the same colour to be in $H$ or not.  $|N_{H'}(S)|$ is also $2$-certifiable. Our product space is $\Omega=(x_1,x_2, \dots x_{|e(G)|})$ where each $x_i=1$ if  the $i$th edge is in $H$. Suppose the current outcome of $H$ is described by $\omega\in \Omega$. Thus if $|N_{H'}(S)|\geq s$ then we can take $I$ to be as follows. Note that for each edge $e$ appearing in $|N_{H'}(S)|$ there exists an edge $e'$ of the same colour or sharing a vertex with $e$ which appears in $H$. We let $I$ to be  the coordinate of all edges $e$ in $N_{H'}(S)$ and coordinates of corresponding $e'$'s. This will guarantee that with respect to any $\omega'$ that agrees with $\omega$ on $I$ must have $|N_{H'}(S)|\geq s$. Thus we can apply Talagrand's inequality with $t=q^2|S|/2$, $r=2$, $c=3$. We  use  that  $60\cdot 3\sqrt{2\cdot 4 q^2 |S|}\leq q^2|S|/2$ since $q\ll 1$  and $|S|\geq q^{-3}$ to get 

$$\Prob[|N_{H'}(S)|> 5q^2|S|] \leq 4e^{-\frac{(q^2|S|/2)^2}{8\cdot 9\cdot 2\cdot 4q^2|S|}}\leq e^{-q^3|S|}.$$
\end{proof}

Finally we are ready prove our first lemma guaranteeing expansion in the second neighbourhood $N_{M,D}^2(S)$ of large sets $S$ for a randomized rainbow matching $M$ and a nearly regular graph $D$.  

\begin{lemma}\label{Lemma_Mnp_second_neighbourhood_expansion}
Let $n^{-1}\ll q \ll p \leq 1$, $n^{-1} \ll \eps <1$ and $q^{-1}\ll d\leq qn/8$. 
Let $G$, $D$ be two bipartite graphs on the same vertex set with bipartition $(X,Y)$ such that $G$ is coloured $(\eps,p,n)$-typical, $\Delta(D)\leq d$ and all but at most $96p^{-2}n/d$ vertices have degrees less than $pd/6$ in $D$.  Let $H$ be obtained from $G$ by picking every edge with probability $q^2/n$, let $M\subseteq H$ be consisting of edges which don't share any vertices or colours with other edges in $H$. If $S\subseteq X$ with $|S|=\frac{n}{d q^3}$ then with probability at least  $1-2e^{-q^6n/d}$  we have $|N^2_{M,D}(S)|\geq (1-6q)n$.
\end{lemma}
\begin{proof}
By Lemma~\ref{Lemma_Gnp_second_neighbourhood_expansion} we have that $|N^2_{H, D}(S)|\leq (1-q)n$ with probability at most  $e^{-q^7|S|}$.  By Lemma~\ref{Lemma_vertex_probabilities} applied with $q_{\ref{Lemma_vertex_probabilities}}=q^2$ we have $|N_{H\setminus M}(S)|\geq 5q^4|S|$ with probability at most $e^{-q^6|S|}$.  By the union bound, we get that with probability at least $1-2e^{-q^6|S|}$ both of these events don't happen.

Since $|N_{H\setminus M}(S)|< 5q^4|S|$ and $\Delta(D)\leq d$, we have that at most $5dq^4|S|$ many edges of $D$ touch $N_{H\setminus M}(S)$. Thus,

 $$|N^2_{M,D}(S)| \geq |N^2_{H,D}(S)| - |N^2_{H\setminus M,D}(S)| \geq (1-q)n-5dq^4|S|=(1-6q)n.$$
\end{proof}

The next lemma builds on Lemma~\ref{Lemma_Mnp_second_neighbourhood_expansion} and guarantees  expansion in the third neighbourhood for subsets of size roughly $n/d^2$.

\begin{lemma}\label{Lemma_expansion_3rd_neighbourhood}

Let $n^{-1}\ll q \ll p \leq 1$, $n^{-1} \ll \eps <1$ and $q^{-1}\ll d \leq \sqrt{n}$. 
Let $G$, $D$ be two bipartite graphs on the same vertex set with bipartition $(X,Y)$ such that $G$ is coloured $(\eps,p,n)$-typical, $\Delta(D)\leq d$ and all but at most $96p^{-2}n/d$ vertices have degrees less than $pd/6$ in $D$.  Let $H$ be obtained from $G$ by picking every edge with probability $q^2/n$ and define $M\subseteq H$ to be consisting of edges which don't share any vertices or colours with other edges in $H$. Then  with probability $\geq 1-e^{-q^{7}n/d}$
the following holds.  

For any $S\subseteq X$ or  $S\subseteq X$ with $|S|\geq \frac{{25}n}{pq^3d}$  there exists $S'\subseteq S$ such that $|S'|= \frac{{24}n}{p q^3d^2}$ such that $|N^3_{D,M}(S')|\geq (1-6q)n$.
\end{lemma}
\begin{proof} Let $J$ be the set of vertices $v\in X\cup Y$ with $d_D(v)\leq pd/ 6$. By assumption $|J|\leq 96p^{-2}n/d\leq \frac{n}{d p q^3}$ since $q\ll p$.  Without loss of generality let us assume $S\subseteq X$. Since $|S|\geq \frac{{25}n}{pq^3d}$  by throwing away at most $\frac{n}{ p q^3 d}$ vertices we may assume $S\cap J=\emptyset$ and $|S|\geq \frac{{24}n}{pq^3d}$. 

By Lemma~\ref{Lemma_regular_bipartite_containers}, there is a set $S'\subseteq S$ with $|S'|\leq {24}n/pq^3d^2$ such that $|N_D(S')|\geq n/ d  q^3$.  By adding extra vertices from $S$ to $S'$ we may assume   $|S'|= {24}n/pq^3d^2$. Fix one such  set $S'$ for each $S$. We say that $S'$ is \emph{bad} if it has $|N^3_{D,M}(S')|< (1-6q)n$. To prove the lemma it is sufficient to show that with probability $\geq 1-e^{-q^{7}n/d}$,  there are no bad sets $S'$. 

Let $S'\subseteq X$  with $|S'|=\frac{{24}n}{pq^3d^2}$ and $|N_D(S')|\geq n/ d   q^3$.  By Lemma~\ref{Lemma_Mnp_second_neighbourhood_expansion} (applied to a subset of $N_D(S')$ of order exactly $n/dq^3$), with probability at least  $1- 2e^{-q^{6}n/d}$, we have
 $|N^2_{M,D}(N_D(S'))|\geq (1- 6q)n.$
Recall that ``$N^3_{D,M}(S')$''  means we are looking at edges going out of $S'$ to be in the order of $D$, $M$ and $D$, thus $N^2_{M,D}(N_D(S')) = N^3_{D,M}(S')$.   So we have shown that $\Prob(S' \text{ is bad})\leq 2e^{-q^{6}n/d}$. By taking a union bound over all $S'\subseteq X$,  with $|S'|=\frac{{24}n}{pq^3d^2}$ and using $d\gg q^{-1}$, we obtain

\begin{align*}
\Prob[\exists \textit{  bad } S']&\leq { \binom{n+n^{-\eps}}{{24}n/p q^3d^2}}  \cdot 2e^{-q^{6}n/d} \leq 2\left(\frac{2en}{ {24}n/pq^3d^2 }\right)^{ \frac{{24}n}{pq^3d^2 }}e^{-q^6n/d}\\
&\leq 2 e^{-q^6n/d + \frac{{24}n}{pq^3d^2}\log{pq^3d^2}} \leq   2e^{-q^6 n/2d}.
\end{align*}
Similarly, the probability that there exists a bad $S'\subseteq Y$ is at most $2e^{-q^6 n/2d}$. Thus with probability at least  $1-e^{-q^7 n/d}$ there are no bad $S'\subseteq X\cup Y$.
\end{proof}

In the next lemma we show that  if we have a coloured regular graph $G$ then if we pick a random rainbow matching and delete all of its edges and colours from the graph $G$ then the remaining graph is still  a coloured regular graph.

\begin{lemma}\label{Lemma_degree_concentration}
Let $n^{-1}\ll   p, q, \eps$ with  $\eps\ll 1$, $q\leq 1/2$, and $p\leq 1$. Let $G$ be coloured $(\eps,p, n)$-regular bipartite graph with bipartition $(X,Y)$. Let $M$ be a random rainbow matching obtained from $G$ by picking every edge with probability  $q/n$ and deleting all colour and vertex collisions.  Let $H$ be $G$ with vertices and colours of $M$ deleted. Then there are numbers $m>n/2, p'>p/2$ such that with  probability at least  $1- e^{-n^{1-\eps}}$, the graph  $H$ is $(\eps/10, p', m)$-regular.
\end{lemma}

\begin{proof} Let $d=pn$ and $\alpha=  \left(1-\frac{q}{n}\right)^{3d}q$. We will see that every edge of $G$ ends up in $M$ with probability roughly $\alpha/n$. Denote $x_H=|X\cap V(H)|$, $y_H=|Y\cap V(H)|$, and $c_H=|C(H)|$. For any vertex $v\in G$, let $d_H(v)=|N_{C(H)}(v)\cap  V(H)|$, and note that for vertices $v\in H$ this is just their degree in $H$. Similarly, for any colour $c\in G$, let $e_H(c)=|E_{G}(c)\cap  V(H)|$, and note that for colours  $c\in C(H)$ this is just the number of edges they have in $H$.
We need to show that with probability at least $1-e^{-n^{1-\eps}}$ the following hold for appropriately chosen $p'$ and $m$:

\begin{itemize}
    \item (P1) $x_H, y_H  = m (1\pm m^{-\eps/10})$,
     \item (P2) $c_H  = m (1\pm m^{-\eps/10})$,
     \item  (P3) $d_H(v) = p'm (1\pm m^{-\eps/10})$, for every $v \in V(H) $,
      \item (P4) $e_H(c) = p'm (1\pm m^{-\eps/10})$, for every $c \in C(H).$
\end{itemize}

\begin{claim} \label{Claim_degree_concentration_expecations}\ 
\begin{itemize}
    \item 
$\E[x_H], \E[y_H] = n(1-p\alpha)(1\pm n^{-\eps/5})$,
\item $\E[d_H(v)], = pn(1-p\alpha)^2(1\pm n^{-\eps/6})$ for every vertex $v\in V(G)$,
\item $\E[c_H]= n(1-p\alpha)(1\pm n^{-\eps/5})$,
\item $\E[e_H(c)] = pn(1-p\alpha)^2(1\pm n^{-\eps/6})$ for every colour $c\in C(G)$,
\end{itemize}
\end{claim}
\begin{proof}
By the symmetry between vertices and colours, it is enough to show that the first two hold.
We estimate several probabilities. At various points,  to bound errors we use that for any positive constant $k$, $(1-q/n)^{\pm kpn^{1-\varepsilon}}=(1\pm n^{-\varepsilon/2})$ and $(1\pm n^{-\eps})^k=1\pm n^{-\eps/2}$ (which holds  as long as $n^{-1}\ll q, \eps, p$). For a colour $c$ edge $xy$, let $F(xy)$ be the set of edges of $G\setminus xy$ sharing a colour or vertex with $xy$. Note that since $G$ is coloured $(\eps, p,n)$-regular, we have $$|F(xy)|=d_G(x)+d_G(y)+|E_G(c)|-3=3d(1\pm 2n^{-\eps}).$$
\begin{itemize}
    \item 
\textbf{The probability that an edge is in M:}\\
We say an edge $e\in G$ is \textbf{chosen} if it was picked at the first step when generating $M$ (with probability $q/n$). By definition of $M$, $e\in M$ exactly when $e$ is chosen and none of the edges from $F(e)$ are chosen. This has probability
\begin{align*}
\Prob[e\in M]&= \frac qn \left(1-\frac qn \right)^{|F(e)|}=\frac qn\left(1-\frac qn\right)^{3d(1\pm 2n^{-\eps})}= \frac{\alpha}{n} (1\pm n^{-\eps/2}). 
\end{align*}

\item 
\textbf{The probability that a pair of edges are both in M:}\\ Let $f, e$ be a pair of edges which don't share any vertices or a colour. Notice that 
$$|F(e)\cup F(f)|= |F(e)|+|F(f)|\pm\Theta(1) = 6d(1\pm3n^{-\eps}).$$

By definition of $M$, we have $e,f\in M$ exactly when $e,f$ are chosen and none of the edges of $F(e)\cup F(f)$ are chosen which happens with probability
\begin{align*}
\Prob[e,f\in M]&= \left(\frac {q}{n}\right)^2 \left(1-\frac qn \right)^{|F(e)\cup F(f)|}=\left(\frac {q}{n}\right)^2\left(1-\frac qn\right)^{6d(1\pm 3n^{-\eps})}= \frac{\alpha^2}{n^{2}}(1\pm n^{-\eps/2}).    
\end{align*}

\item 
\textbf{The probability that a vertex/colour is in M:}\\ For any $v$, we have 
\begin{align*} \Prob[v\in M] &= \sum_{y\in N_G(v)}\Prob[vy\in M]=d_G(v)\frac{\alpha}{n}(1\pm n^{-\eps/2})=p\alpha  (1\pm n^{-\eps/2})^2= p\alpha(1\pm n^{-\eps/4})\end{align*}
By the symmetry between vertices and colours, we also have $\Prob[c\in C(M)]=p\alpha(1\pm n^{-\eps/4})$ for every colour $c$.

\item 
\textbf{For an edge $uv$, the probability that $u$ or $c(uv)$ is in $M$:}\\ Fix an edge $uv$. We first estimate the probability that ``$u\in M$ and $c(uv)\in M$''. Notice that there are two ways this can happen --- either $uv\in M$ or there are three distinct vertices, $w,x,y$ such that $uw$, $xy\in M$ with $c(xy)=c(uv)$. We will see that the  probability of the first event is negligible compared to the second. For an edge $uv$, let $J(uv)\subseteq E(G)\times E(G)$ be the set of pairs $(uw,xy)$ as described above. We have $$ (d_G(u)-1) (|E_G(c(uv))|- 2)\leq  |J(uv)| \leq d_G(u) |E_G(c(uv))|.$$
This implies $|J(uv)|= p^2n^2(1\pm n^{-\eps})^3$ which implies
\begin{align*}\Prob[u\in M, c(uv)\in M] &= 
 \Prob[vu\in M]+\sum_{(e,f)\in J(uv)}{\Prob[e,f\in M]} \\
&=\frac{\alpha}{n}(1\pm n^{-\eps/2})+(pn)^2 \frac{\alpha^2}{n^{2}}(1\pm n^{-\eps/2})^4 =p^2\alpha^2  (1\pm n^{-\eps/5}), \end{align*}
where in the last equality we used that $\alpha \approx e^{-3pq} q$ and so $\alpha/n \ll p^2\alpha^2 n^{-\eps/2}$, as long as $n$ is sufficiently large. Thus,

\begin{align*}\Prob[u\in M \textit{ or } c(uv)\in M]  &= \Prob [u\in M] + \Prob[ c(uv)\in M] - \Prob[u\in M, c(uv)\in M] = \left(2 p \alpha   - p^2\alpha^2 \right)(1\pm  n^{-\eps/5} ) \end{align*}
\end{itemize}
Finally we are ready to estimate the expectations in the claim. By  linearity of expectation, 
\begin{align*}\E[x_H]&= \sum_{v\in X}\Prob(v\not\in M)=n(1\pm n^{-\eps}) (1-p\alpha(1\pm n^{-\eps/4}))= n(1-p\alpha)(1\pm n^{-\eps/5}).\\
 \E[d_H(v)]&= \sum_{u\in N_G(v)}{(1-\Prob[u\in M \textit{ or } c(vu)\in M]})\\
& =  pn(1\pm n^{-\eps})(1-(2 p \alpha   + p^2\alpha^2)(1\pm 
n^{-\eps/5})) =  pn(1-p\alpha)^2(1\pm n^{-\eps/6})
\end{align*}
\end{proof}
Fix $m=(1-p\alpha)n$ and $p'= p(1-p\alpha)$. Notice that we may assume $m>n/2$ and $p'>p/2$ as we can guarantee $p\alpha < pq \leq 1/2$, since $n$ is sufficiently large.

Notice that the random variables $x_H$, $y_H$, $c_H$, $d_H(v)$, and $e_H(c)$ depend on the probability space $\Omega=\{0,1\}^{E(G)}$ with every coordinate being $1$ with probability $q/n$. All these variables are $3$-Lipshitz. Set $\sigma^2= 3^2\sum_{e\in E(G)}q/n(1-q/n)$ and notice that $$\frac{9pqn}{4}< \sigma^2 =  9pqn (1-q/n) (1\pm n^{-\eps})^2 <  10pqn.$$

Set $t=n^{1/2-\eps/3}$. Note that $t \leq 2\sigma/3$ and $t\sigma < 3n^{1-\eps/3}$, since $n$ is sufficiently large. By Lemma~\ref{Lemma_Azuma_variance} we have
$$\Prob{\left[|x_H-\E[x_H]| > 3n^{1-\eps/3}  \right]} 
< 2e^{-\frac{n^{1-\frac{2\eps}{3}}}{4}} < \frac{e^{-n^{1-\eps}}}{n^3}.$$

Similarly one can show that each of the random variables $y_H$, $c_H$, $d_H(v)$, and $e_H(c)$ are within $ 3n^{1-\eps/3}$ of their expectations with probability at least $1-\frac{e^{-n^{1-\eps}}}{n^3}$ . If we take a union bound over all vertices and colours, we can guarantee that $x_H$, $y_H$, $c_H$, $d_H(v)$, and $e_H(c)$ are simultaneously all  within $ 3n^{1-\eps/3}$ of their expectations for all $c$ and $v$. To conclude that (P1)-(P4) hold, it remains to check that $3n^{1-\eps/3} +  mn^{-\eps/6} \leq m^{1-\eps/10}.$ 
\end{proof}

We now prove the main result of this section.

\begin{proof}[Proof of Lemma~\ref{Lemma_main_expansion}]
Fix $\hat q=(25p^{-1}q^{4})^{1/3}$. Note that $\hat{q} \ll p$ since $q\ll p$.  Suppose $G$ has bipartition $(X,Y).$ Let $M_0$ be generated by picking every edge of $G$ with probability  $\hat q^2/n$ and deleting all  vertex and colour collisions.  Let $H$ be obtained from $G$ by removing the  vertices and colours of $M_0$.  By Lemma~\ref{Lemma_degree_concentration}, with probability at least  $1-e^{-n^{1-\eps}}$ we have that  $H$ is $(\eps/10, p',m)$-regular for some suitable $p'$ and $m$.  
When $H$ is $(\eps/10, p',m)$-regular, by  Lemma~\ref{Lemma_variant_of_MPS_nearly_perfect_matching}, there is a rainbow matching $M_1$ in $H$ of size $\geq m-m^{1-2\gamma}$.  Note that $|V(H)\cap X|= |X|-|M_0|$  and when $H$ is  $(\eps/10, p',m)$-regular we have
$|V(H)\cap X| = m(1\pm m^{-\eps/10})$. Therefore, it follows that $m\geq |X|- |M_0|-m^{1-\eps/10} \geq n- |M_0|-n^{1-\eps}-n^{1-\eps/10} $, which implies that $|M_1\cup M_0|\geq m-m^{1-2\gamma}+|M_0|\geq n -2n^{1-\eps/10} - n^{1-2\gamma} \geq n -2n^{1-2\gamma} \geq n -n^{1-\gamma} $, since $\gamma \ll \eps$.  We will show that the conclusion of the lemma holds for the randomized rainbow matching $M=M_0\cup M_1$. Notice that with probability  at least $1-e^{-n^{1-\eps}}$ we  have 
\begin{itemize}
    \item [$\mathcal{E}_1$:]\label{Event_Matching_Large} $|M|\geq (1-n^{-\gamma})n$.
\end{itemize}
 
Let $D$ be a bipartite graph with the same bipartition as $G$ having $\Delta(G)\leq d$ and at most $\leq 96p^{-1} n/d$ vertices of degrees  less than $ pd/6$. We can apply Lemma~\ref{Lemma_expansion_3rd_neighbourhood} to $M_0$, $G$, and $D$  and obtain that with probability at least $1- e^{-\hat q^{7}n/d}$ we have
\begin{itemize}
    \item [$\mathcal{E}_2$:] \label{Event_Expansion} For  any $S\subseteq X$ or  $S\subseteq Y$ with $|S|\geq \frac{25n}{d p \hat q^3}$  there exists $S'\subseteq S$ such that $|S'|= \frac{24n}{p \hat q^3d^2}$ such that $|N^3_{D,M_0}(S')|\geq (1-6\hat q)n$.
\end{itemize}
So with probability at least $1-e^{-n^{1-\eps}} - e^{-\hat q^{7}n/d} \geq 1-2e^{-n^{1-\eps}}$ both events $\mathcal{E}_1$ and $\mathcal{E}_2$ happen. Clearly (i) holds then, let us show that (ii) holds as well.  Let $D$ be as earlier and without loss of generality assume  $S\subseteq X$ with $|S|\geq \frac{n}{d  q^{4}}$.  Then since $q\ll p$ it follows that $|S|\geq \frac{25n}{d p \hat q^3}$. Since $\mathcal{E}_2$ holds there exists $S'\subseteq S$ such that $|S'|= \frac{24n}{p \hat q^3d^2} \geq 24 n/25q^4d^2$ such that $|N^3_{D,M_0}(S')|\geq (1-6\hat q)n$. Since $|X\setminus M|\leq n+n^{1-\eps}-(n-n^{1-\gamma}) \leq 2n^{1-\gamma}$, it follows that 
 $$|N^4_{D, M}(S')| \geq |N^3_{D,M_0}(S')| -|X\setminus M|\geq(1-  6\hat q)n- 2n^{1-\gamma} \geq  (1- q)n,$$
where the last inequality holds since $q\ll p$. Finally we can always add extra vertices of $S\setminus S'$ to $S'$ to make it exactly of size $n/q^4d^2$. This finishes the proof.
\end{proof}

\subsection{Switchings via Expansion}
\label{sec:switching}
The main result of this section is the lemma below which is our main tool for doing switchings.  It says that if we have two expanders $(D_1, M)$ and $(D_2,M)$ such that $D_1$ and $D_2$ are two bipartite graphs on the same vertex set then almost all pairs of vertices lying in the opposite sides of the bipartition of the graph $D_1\cup D_2\cup M$ have a short rainbow $(D_1\cup D_2)-M$-alternating path between them.

\begin{lemma}\label{Lemma_expansion_paths} Let $d^{-1/2}\leq A^{-1}\leq\eps/100\ll 1$ and further, $d\log{d}\geq 8A^2\log{n}$. Suppose we are given two bipartite graphs $D_1, D_2$ and a rainbow matching $M$ on the bipartition $(X,Y)$ with $\Delta(D_1), \Delta(D_2)\leq d$,   $M\cup D_1\cup D_2$  properly edge-coloured, $C(M)$, $C(D_1)$, $C(D_2)$  pairwise disjoint, and $|X|, |Y|< (2-4\eps)n$. If for both $i=1,2$,  $(D_i,M)$ is a $(d,A, \eps,n)$-expander then there is a set $B\subseteq X\cup Y$ of at most $4An/d$ vertices, such that for all $u,v\not\in B$ lying in the opposite sides of the bipartition, there is a  $(D_1\cup D_2)$-$M$-alternating rainbow path  from $u$ to $v$ of  length at most $8\lceil \frac{\log{n}}{\log{(d/4A)}} \rceil$.
\end{lemma}

The alternating paths found by the above lemma will be used to go from one rainbow matching to another. When proving the existence of large rainbow matchings in typical graphs, we will start from some rainbow matching and iteratively do such switchings, eventually enlarging the original rainbow matching to one of a desired size. We need a simple lemma which claims that if we have an expander $(D,M)$ then any ``small'' perturbation of the matching $M$ will keep the expansion properties. (In applications of this lemma ``small'' would mean sub-polynomial.)

\begin{lemma}\label{aux:exp1toexp2} Suppose we are given two  matchings $M_1, M_2$ and a bipartite graph $D$ with $\Delta(D)\leq d$, all on the same bipartition $(X,Y)$. If  $(D, M_1)$ is a $(d, A, \eps,n)$-expander and $|M_1\triangle M_2|< \eps n/10 d^2$ then  $(D, M_2)$ is a $(d, A, 2\eps, n)$-expander. \Alexey{[Alexey: the previous version of this which allowed the two graphs to have different vertex sets didn't prove the condition $X>An/d$, $Y>An/d$ for $(M_2, D)$.]}
\end{lemma}
\begin{proof} 
Let $S\subseteq X$ or  $Y$ with $|S|\geq An/d$.  
Since $(D, M_1)$ is a $(d, A, \eps,n)$-expander, there is a subset $S'\subseteq S$ with $|S'|=An/d^2$ such that $|N^4_{D, M_1}(S')|\geq (1-\eps)n$. By definition, for each $v\in N^4_{D, M}(S')$ there is a $D$-$M_1$-alternating path of length four from $S'$ to $v$. We call such an alternating $D$-$M_1$ path of length four \emph{bad} if it uses any vertex from $V(M_1)\setminus V(M_2)$  and \emph{good} otherwise. The number of bad paths is at most $5\Delta(D)^2 |V(M_1)\setminus V(M_2)|\leq 10|M_1\triangle M_2|d^2 < \eps n$.
Indeed, there are $5$ ways to choose which vertex of the path of length four is in $V(M_1)\setminus V(M_2)$ and then at most $\Delta(D)^2$ ways to choose $D$-$M_1$-alternating path which has this vertex in the correct position.
Hence there are at least $(1-2\eps)n$ good paths. Every good path is an alternating $D$-$M_2$  path, and so we have $|N^4_{D, M_2}(S')|\geq (1-2\eps)n$.
\end{proof}

 We say a path $P$ in an edge-coloured graph $G$ \emph{avoids} a vertex subset $V'\subseteq V(G)$ if it doesn't contain any vertex from $V'$. Similarly, $P$  \emph{avoids} a colour subset $C'\subseteq C(G)$ if it does not contain any edge of colours from $C'$. To prove Lemma~\ref{Lemma_expansion_paths} we first show that given an expander $(D, M)$ such that  $D$ and $M$ are colour disjoint, all large sets $S$ ``expand'' in the following coloured fashion: almost every vertex in $v\in V(D)\cup V(M)$ can be reached from some $s\in S$ via a rainbow $D-M$-alternating path of length four and additionally, this path avoids some small set of forbidden colours and vertices prescribed to $s$ a priori (Lemma~\ref{Lemma_expansion_conversion}).  Then we apply this iteratively  to obtain a similar expansion property for smaller sets (Lemma~\ref{Lemma_expansion_recursion}). Finally via applying this iteration multiple times we show that almost all vertices can reach almost all vertices via rainbow $D-M$-alternating paths of length $O(\log{n}/\log{\log{n}})$. (Lemma~\ref{Lemma_expansion_one_vertex}).

\begin{lemma}\label{Lemma_expansion_conversion} 
Let $d^{-1/2}\leq A^{-1}\leq\eps/100\ll 1$. Suppose we are given a bipartite graph $D$ with $\Delta(D)\leq d$,  $M$ a rainbow matching such that $M\cup D$ has bipartition $(X,Y)$, is properly edge-coloured, and $C(M)$ and $C(D)$ are disjoint. Let $C=C(D)\cup C(M)$, $V=V(D)\cup V(M)$. If $(D,M)$ is a $(d,A, \eps,n)$-expander, then for any  $S\subseteq X$ or $Y$ with $|S|=An/d$ and any collections of ``forbidden'' colours and vertices $\{C(s)\subseteq C|s\in S\}$,  $\{V(s)\subseteq V|s\in S\}$ with $|C(s)|\leq  A^{-2}d$, $|V(s)|\leq  A^{-2}d$, $s\notin V(s)$ for all $s\in S$ the following holds. There are at least $(1-2\eps)n$ vertices  $v\in V$ for which there is a $D$-$M$-alternating rainbow path $P_v$ of length four going from some $s_v\in S$ to $v$ and avoiding $C(s_v)$ and $V(s_v)$.
\end{lemma}
\begin{proof} 
Since $(D,M)$ is a $(d,A, \eps,n)$-expander there exists  $S'\subseteq S$ of order $An/d^2$ such that  $|N^4_{D, M}(S')|\geq (1-\eps)n$. For every $v\in N^4_{D, M}(S')$, there is some $s\in S'$ and a $D$-$M$-alternating path $P_{v}$ of length four going from $s$ to $v$.  We say that $P_{v}$ is \emph{bad} if either $P_{v}$ is not rainbow, or $P_{v}$ doesn't avoid $C(s), V(s)$. We say $P_{v}$ is good otherwise.  We will show that the total number of  bad paths among all the paths $\{P_{v}\}_{v\in V}$ is at most $\eps n$. Note that this  would be enough for the conclusion of the lemma as we can take the final vertex set to be $\{v\in N^4_{D, M}(S')|P_{v}\textit{ is good}\}$. To count the total number of bad paths $P_v$, we count for all $s\in S'$   how many bad paths start at $s$.
 
Fix a vertex $s\in S'$. Notice that a bad $D$-$M$ alternating path $sxyzw$ must satisfy at least one of the following:
\begin{itemize}
\item $sxyzw$ is not rainbow. Since  $D\cup M$ is properly edge-coloured, $M$ is rainbow and colour disjoint from $D$, this happens only when $c(sx)=c(yz)$. There are at most $\Delta(D)\leq  d$ such $D$-$M$ alternating paths starting from $s$, since there are at most $d$ choices for $x$ and at most one for all the other vertices.
\item some vertex among $x,y,z,w$ is from $V(s)$. The number of such $D$-$M$ alternating paths is at most $4|V(s)|d$. Indeed, there are at most four choices which vertex among $x,y,z,w$ is in $V(s)$. Then $|V(s)|$ choices to specify that vertex $v$. Now suppose $x=v$ then there is at most one choice for $y$ depending if $x$ is covered by the matching $M$ or not, at most $d$ choices for $z$ since $\Delta(D)\leq d$ and finally at most one choice for $w$. So there are at most $d$ such paths with $x=v$. Similar argument applies if $y=v$ or $z=v$ or $w=v$.
\item some edge among $sx$, $xy$, $yz$, $zw$ has a colour appearing in $C(s)$. The number of such $D$-$M$ alternating paths is at most $4|C(s)|d$. Again, there are $4|C(s)|$  choices to specify which one of the edges $sx$, $xy$, $yz$, $zw$ has colour  $c\in C(s)$.  Suppose we are counting the number of paths $sxyzw$ with $c(sx)=c$. Since the colouring is proper there is at most one edge of colour $c$ coming out of $s$ therefore  at most one choice for vertex $x$. Then there is at most one choice for $y$, depending if $x$ is covered by the matching $M$ or not, at most $d$ choices for $z$ since $\Delta(D)\leq d$ and finally at most one choice for $w$. A similar analysis with the same bound will apply if $c(xy)=c$, $c(yz)=c$ or $c(zw)=c$.
\end{itemize}
Thus  the total number of bad paths starting at $s$ is $\leq 8A^{-2}d^2+d\leq 9A^{-2}d^2$ (using $d^{-1/2}\leq A^{-1}$). Summing over all $s\in S'$, we get that the total number of bad paths is at most $9A^{-2}d^2|S'|= 9A^{-1}n\leq \eps n$ as desired. 
\end{proof}

Recall that $N^{t}_{D,M}(S)$ denotes the set of vertices to which there is a $D$-$M$-alternating path of length $t$ starting in $S$.
We use  $\hat N^t_{G,H}(S)$ to denote the set of of vertices to which there is a $D$-$M$-alternating \emph{rainbow} path of length $t$ starting in $S$. 

\begin{lemma}\label{Lemma_expansion_recursion}
Let $d^{-1/2}\leq A^{-1}\leq\eps/100\ll 1$ and $t\leq A^{-2}d/4$. Suppose we are given a bipartite graph $D$ with $\Delta(D)\leq d$,  $M$ a rainbow matching such that $M\cup D$ has bipartition $(X,Y)$, is properly edge-coloured, and $C(M)$ and $C(D)$ are disjoint. If $(D, M)$ is a $(d,A, \eps,n)$-expander then for every set of vertices  $S\subseteq X$ or $Y$  with  $|\hat N^{4t}_{D,M}(S)|\geq (1-2\eps)n$ there is $S'\subseteq S$ with $|S'|=   \lceil 2A|S|/d \rceil$  and $|\hat N^{4t+4}_{D,M}(S)|\geq (1-2\eps)n$.
\end{lemma}
\begin{proof}
For each $v\in \hat N^{4t}_{D,M}(S)$, by definition there exists $s_v\in S$ and a $D$-$M$-alternating rainbow path $P_{v}$  of length $4t$ from  $s_v$ to $v$. For each $v$ fix such $s_v$ and $P_v$. For each $s\in S$, let $p(s)$ be the number of paths $P_v$ starting at $s$. 
We have $\sum_{s\in S}p(s)=|\hat  N^{4t}_{D,M}(S)|\geq (1-2\eps)n$. Let $S'\subseteq S$ be a  subset of size $\lceil2A|S|/d\rceil$ with $\sum_{s\in S'}p(s)$ maximum.
By averaging, $\frac{1}{|S'|}\sum_{s\in S'}p(s)\geq  \frac{1}{|S|}\sum_{s\in S}p(s)$.
Thus  $|\hat N^{4t}_{D,M}(S')|\geq \sum_{s\in S'}p(s)\geq \frac{|S'|}{|S|}\sum_{s\in S}p(s)\geq (1- 2\eps)2An/d\geq An/d$.

Let  $T\subseteq \hat N^{4t}_{D,M}(S')$  be a subset of size exactly $An/d$. To each vertex $v\in T$,  assign forbidden sets of colours and vertices $C(v):=C(P_v)$, $V(v):=V(P_v)\setminus\{v\}$, and note  that $|C(v)|, |V(v)|\leq 4t\leq A^{-2}d$. By Lemma~\ref{Lemma_expansion_conversion},  we get $(1-2\eps)n$ vertices $u$ together with rainbow  $D-M$-alternating paths $Q_u$ of length four from  some $v_u\in T$ such that $Q_u$ avoids $C(v_u)$ and $V(v_u)$. It is easy to check that for each $u$, $P_{v_u}\cup Q_u$ is a rainbow $D-M$-alternating path of length $4t+4$.  This finishes the proof.
\end{proof}
 
\begin{lemma}\label{Lemma_expansion_one_vertex}
Let $d^{-1/2}\leq A^{-1}\leq\eps/100\ll 1$ and further, $d\log{d
}\geq 8A^2 \log{n}$. Suppose we are given a bipartite graph $D$ with $\Delta(D)\leq d$,  $M$ a rainbow matching such that $M\cup D$ has bipartition $(X,Y)$, is properly edge-coloured, and $C(M)$ and $C(D)$ are disjoint. If $(D, M)$ is a $(d,A, \eps,n)$-expander then  for  $t=\lceil\frac{\log n}{\log{(d/4A)}}\rceil$, all but possibly at most $2An/d$ vertices $v\in V(M)\cup V(D)$ satisfy $$|\hat N^{4t}_{D,M}(v)|\geq (1-2\eps)n.$$
\end{lemma}
\begin{proof}
Suppose the lemma is false. Then without loss of generality,
there are  at least $An/d$  vertices  $v\in X$ such that $|\hat N^{4t}_{D,M}(v)|< (1-2\eps)n$. Let . Let $S_0\subseteq X$ be this set of vertices. 

Choose $S_1\subseteq S_0$ to be of size exactly $\lceil An/d \rceil$. Then by Lemma~\ref{Lemma_expansion_conversion}  it follows that  $|\hat N^4_{D,M}(S_1)|\geq (1- 2\eps)n$ (we assign $C(s)=V(s)=\emptyset$ for all $s\in S_1$).
Now we can iteratively apply Lemma~\ref{Lemma_expansion_recursion}  and obtain sets $S_1 \supseteq S_2\supseteq \dots \supseteq S_t$ with $|\hat N^{4i}_{D,M}(S_i)|\geq (1-2\eps)n$ and $|S_i|= \lceil 2 A|S_{i-1}|/d\rceil$ such that $|S_t|=1$. Indeed, note that $|S_1|\leq 4An/d$ and for all $i\geq 2$, $|S_i|\leq 4 A|S_{i-1}|/d$. Moreover,  the iterative steps can be applied because $i\leq t \leq 2\log{n}/\log{d}\leq A^{-2}d/4$.
Therefore for $t= \lceil\frac{\log n}{\log {(d/4A)}}\rceil$  we must have $|S_t| \leq 1$, but since $S_i$ is always non-empty we have $|S_t|=1$. Thus there is a vertex $s\in S_t$ with $|\hat N^{4t}_{M,D}(s)|\geq (1-2\eps)n$, contradicting the definition of $S_0$.
\end{proof}

We now prove the main lemma of this section.
\begin{proof}[Proof of Lemma~\ref{Lemma_expansion_paths}:] 
Set $t=\lceil\frac{\log n}{\log d/4A}\rceil$. By applying Lemma~\ref{Lemma_expansion_one_vertex}  first with $(D_1,M)$ and  then with $(D_2,M)$ we obtain is a set  $B$ of order at most $4An/d$ such that all vertices outside $B$ have $|\hat N^{4t}_{D_i,M}(v)|\geq (1-2\eps)n$ for $i=1,2$. Now let $u\in X,v\in Y$ be two vertices  outside $B$, then  $|\hat N^{4t}_{D_1,M}(u)|\geq (1-2\eps)n$ and $|\hat N^{4t}_{D_2,M}(v)|\geq (1-2\eps)n$. Notice that $\hat N^{4t}_{D_1,M}(u)\subseteq X\cap V(M)$ and $\hat N^{4t}_{D_2,M}(v)\subseteq Y\cap V(M)$ (since these sets are defined by \emph{even} length alternating paths from $u$ and $v$). For any $x\in N^{4t}_{D_1,M}(u)$, by definition, there is a rainbow $D_1-M$-alternating path going from $u$ to $x$ of length $4t$ whose last edge is in $M$, call this path $P_{ux}$. Similarly, for any $y\in N^{4t}_{D_1,M}(v)$, there is a rainbow $D_2-M$-alternating path going from $v$ to $y$ of length $4t$ whose last edge is in $M$, call this $P_{vy}$. We claim that  there is a pair $x\in,y\in Y$ such that $xy\in M$ is the last edge of $P_{ux}$ and $P_{vy}$. Indeed, otherwise we will have $ (2-4\eps)n \leq |\hat N^{4t}_{D_1,M}(u)|+|\hat N^{4t}_{D_2,M}(v)| \leq |M| \leq |X|, |Y|< (2-4\eps)n$, a contradiction.
So $P_{ux} \cup P_{vy}$ is a rainbow walk which must contain a  rainbow $(D_1\cup D_2)$-$M$-alternating path from $u$ to $v$  of length at most $8t$. (In fact, since $u,v$ lie in the opposite sides of the bipartition, this path must be of odd length).
\end{proof}

\section{Large matchings in coloured typical graphs}
In this section, we combine previous ones to show that typical graphs have large rainbow matchings. We prove the following technical theorem which will imply all our other theorems.  
\begin{theorem}\label{Theorem_typical_technical}
Let $n^{-1}\ll k^{-1}\ll p\leq 1$, $n^{-1} \ll \eps<1$ and fix $d=\frac{k\log n }{\log\log n}$. Suppose that we have graphs $G\subseteq H$ with the following properties:
\begin{itemize}
    \item $H$ is properly edge-coloured, bipartite with bipartition $(X,Y)$ such that $|X|=|Y|=n$, and vertex $v\in V(H)$ has $|N_H(v)\cap V(G)|\geq 0.3pn$.
    \item $G$ is coloured $(\eps, p, n)$-typical with at least $n+6d$ colours.
\end{itemize}
Then $H$ has a rainbow perfect matching.
\end{theorem}

\noindent
The full power of the above theorem will only be used to prove our results about generalized Latin arrays. For our results about Latin squares and Steiner systems, a weakening of this result stated as Corollary~\ref{Corollary_typical_matching} will be sufficient.

The above theorem is proved using the approach described in the Introduction (see Section~\ref{subsec:nibble},  (S1) -- (S4)). Note that there are two graphs in the assumption of Theorem~\ref{Theorem_typical_technical}. Inside the typical graph $G$, we find a randomized rainbow matching $M$ of size $n-n^{1-\gamma}$ with  expansion properties with respect to any collection $D$ of $d$ colours of the graph $G$ . By these expansion properties we know that almost all $x\in X$ and $y\in Y$ have short rainbow $D-M$-alternating paths between them in $G$ (Lemma~\ref{Lemma_expansion_paths}). Below  think of $D$ being some subset of unused colours on $G$, sometimes these colours can come from the graph $H$. We iteratively increase the size of the matching until we get a perfect rainbow matching in $H$.
At each step we obtain a rainbow matching $M^i$ of size $|M^{i-1}|+1$, such that the edit distance between  each $M^i$ and $M$ is still sufficiently small. This guarantees that the expansion properties that $M$ originally had are still preserved for $M^i$. Note that since $G$ has 
at least $n+6d$ colours we always have at least $\Omega(\log{n}/\log{\log{n}})$ unused colours outside of $M$. 
 
For rather technical reasons, to perform switchings, we first randomly split $G$ into three graphs $G_1, G_2, G_3$ and find  randomized large matchings $M_1\subseteq G_1, M_2\subseteq G_2, M_3\subseteq G_3$ in these subgraphs using Lemma~\ref{Lemma_main_expansion}. We set $M=M_1\cup M_2\cup M_3$ to get a matching of size $n-n^{1-\gamma}$.  The advantage of splitting like this is that it now gives us three disjoint matchings with expansion properties which will be useful for finding \emph{disjoint} alternating paths/cycles for switching purposes. Such alternating paths are found using Lemma~\ref{Lemma_expansion_paths}. The way we use alternating paths/cycles for enlarging the matching  is illustrated in Figure~\ref{Figure_switching}. The idea is to first fix two vertices $x_0, y_0$ which we want to add to the matching.  Imagine $x_0$ and $y_0$ had edges $x_0y_1'$ and $y_0x_1'$ going to $M_1$ of colours still unused on $M=M_1\cup M_2\cup M_3$. Assume $x_1'$ is matched to $y_1$ in $M_1$ and $y_1'$ is matched to $x_1$in $M_1$. Suppose between  $x_1$ and $y_1$ we could find an alternating $D-M$ rainbow path $P$ (this is true for \emph{almost all} $x_1$ and $y_1$). Then we could switch the $M$-edges to non-$M$-edges and vice versa on the path $P\cup\{x_0y_1',x_1y_1', y_0x_1', x_1'y_1\}$. This would increase the size of the matching $M$ by one immediately. However, we cannot guarantee that such $x_1',y_1'$ will be present in $G_1$, that is, such that $c_2:=c(x_0y_1'),c_3:=c(y_0x_1')\notin C(M)$. But we can always guarantee that a choice of $x_1'$ and $y_1'$ will be present such that $c_2$ appears on $M_2$ and $c_3$ appears on $M_3$. So then if say $x_2y_2$ is the $c_2$-edge in $M_2$, by the expansion properties of $M_2$ we know there is a $D-M$-alternating rainbow path between $x_2$ and $y_2$ which together with the edge $x_2y_2$induces an alternating cycle along which if we ``switch'', that is, we make all the non-edges of $M_2$ edges of $M_2$  and vice versa, $M_2$ remains to be rainbow. We apply the same argument to kick out the colour $c_3$ from $G_3$. Now colours $c_2$ and $c_3$ become available to use in the matching, and thus we can do the aforementioned switching along the edges of the path $P\cup\{x_0y_1'\}\cup\{y_0x_1'\}$, thus increasing the matching $M$ by size one. Finally note that at each step $|M^i\triangle M^{i+1}|= O(\log{n}/\log{\log n})$, since we switch along at most three paths/cycles of $O(\log{n}/\log{\log n})$ length. Because of this after at most $O(n^{1-\gamma})$ steps, $|M^{i}\triangle M| \leq  O(n^{1-\gamma}\log{n}/\log{\log{n}}) \ll |M|$, thus $M^i$ will still have the expansion properties, therefore we can iterate this approach by having $M^i$ instead of $M$.
 
\begin{proof}
We can assume that $\eps\ll 1$ since any $(\eps, p, n)$-typical graph is also $(\eps', p, n)$-typical for all $\eps'<\eps$. Choose  auxiliary constants $q=k^{-1/9}$, $\gamma$ satisfying $n^{-1}\ll \gamma \ll \eps$. 

 We call colours of $G$ \emph{large}. Notice that by coloured  $(\eps, p, n)$-typicality of $G$, large colours have $(1\pm n^{-\eps})n$ edges and so there are  less than $n^{1+\eps}$ large colours.  Denote $C_0=C(H\setminus G)$.
 
Partition $V(G)$ and $C(G)$ into three sets $V_1, V_2, V_3$ and $C_1$, $C_2$, $C_3$  with each vertex/colour ending up in each $V_i/C_i$ independently with probability $1/3$. For $i,j=1,2,3$ we let $G_{i,j}$ to be the subgraph induced by the vertex set $V_i$ and by the colour set $C_j$. We also denote by $G_i=G_{i,i}$.
 
\begin{claim}\label{Claim_Gi_properties}
With positive probability
\begin{enumerate}[(i)]
    \item  For all $i,j\in \{1,2,3\}$, $G_{i,j}$ is $(\eps/8, p/3, n/3)$-typical.
    \item  For all $i,j\in \{1,2,3\}$, every vertex $v\in V(H)$ satisfies  $|N_{C_i\cup C_0}(v)\cap V_j|\geq p n/40$.
\end{enumerate}
\end{claim}
\begin{proof}
By Lemma~\ref{Lemma_random_subgraph_typical_graph} (a) applied with $q_{\ref{Lemma_random_subgraph_typical_graph}}=1/3$,  with  probability  at least $1-e^{-n^{1-\eps/2}}$, $G_i$ is $(\eps/8, p/3, n/3)$-typical for each $i=1,2,3$.  Thus (i) holds with probability at least $1-9e^{-n^{1-\eps/2}}$.

Property (ii) follows from Chernoff's bound. Fix a vertex $v\in V(H)$. Then every $y\in N_H(v)\cap V(G)$ is in  $N_{C_i\cup C_0}(v)\cap V_j$ independently with probability $\geq 1/9$. Indeed, if $c(vy)\not\in C(G)$ this happens when $y\in V_j$ which has probability $1/3$. When $c(vy)\in C(G)$, then this happens when both $y\in V_j$ and $c(vy)\in C_i$ which has probability $1/9$. So we get $\E[|N_{C_i\cup C_0}(v)\cap V_j|]\geq pn/30$.
So by Chernoff's bound, the probability that $|N_{C_i\cup C_0}(v)\cap V_j|]< p n/40$ is less  than $e^{-p n/960}$. Thus, with probability at least $1-9n^{-1}e^{-p n/960}$, (ii) holds.
\end{proof}

 \begin{claim}\label{Claim_Mi_properties}
 For each $i=1,2,3$, there is a rainbow matching $M_i$ in $G_i$ such that the following hold:
 \begin{enumerate}[(a)]
     \item  $|M_i|\geq n/3-n^{1-\gamma}$ 
     \item  For every set of $d$ large colours $D$,  define  $D_i$  to be the subgraph induced by edges on $V_i$ which have colours from $D$. Then the pair $(D_i, M_i)$ is a $(d,  q^{-4}, q, {n/3})$-expander.
 \end{enumerate}
 \end{claim}
\begin{proof}
 Fix $i=1,2,3$. Let $D$ be a set of $d$ large colours. By the pigeonhole principle there exists some $j\in \{1,2,3\}$ such that $|D\cap C_j|\geq d/3$. Since $G_{i,j}$ is coloured $(\eps/8, p/3, n/3)$-typical it follows from Lemma~\ref{lem:fewisolated} that  $G_{i,j}[D\cap C_j]$ has at most $32(p/3)^{-2}(n/3)/(d/3) = 96(p/3)^{-2}(n/3)/d$ vertices of degree $\leq (p/3) (d/3)/2 = pd/18$, which implies that so does $D_i$ (note that $D_i$ potentially has more colours but that can only increase the degrees of vertices). Therefore by Lemma~\ref{Lemma_main_expansion} (applied to $D$ and $G_i$) with probability at least  $1-2e^{-n^{1-\eps/8}}$  (i) and (ii) hold with respect to $D_i$.  Thus, by the union bound, the probability that $|M_i|< n/3-n^{1-\gamma}$  or there exists some  $D_i$ which is not 
   $(d,q^{-4}, q, {n/3})$-expander  is at most ${\binom{|C(G)|}{d}}2e^{-n^{1-\eps/8}} \leq  (n + n^{1-\eps})^d \cdot 2e^{-n^{1-\eps/8}}  \ll 1$.
\end{proof}

Let $M= M_1\cup M_2\cup M_3$.  We claim that $M$ can be ``extended'' to a perfect rainbow matching $M'$  of $H$ using colours of $C_0\cup (C(G)\setminus C(M))$ such that $|M'\triangle M|\leq 98n^{1-\gamma}\frac{\log n}{\log{d}}$.  Indeed, pick $r$ largest such that $|M'|=|M|+r$, $M'$ is rainbow and   $|M'\triangle M|\leq 49r\frac{\log n}{\log{d}}$. If $M'$ is not a perfect matching,  then there exist vertices $x_0\in X,y_0\in Y$  outside of $M'$.   From Claim~\ref{Claim_Mi_properties} (a), we have $|M|\geq 3(n/3-n^{1-\gamma})$ which gives $r\leq 3n^{1-\gamma}$.

First of all, note that $M$ must be missing at least $6d$ many  large colours (since there are at least $n+6d$ large colours in total).  For $j=1,2,\dots,6$, let $D^j$  be  disjoint collections of such large colours each of size $d$.  Note that since these colours are large, Claim~\ref{Claim_Mi_properties} tells us that for every $i=1,2,3$ and $j=1,2,\dots, 6$   the  pair $(D_i^j,M_i)$ is a $(d, q^{-4}, q, n/3)$-expander.  

Denote $M_i':=M_i\cap M'$, for all $i=1,2,3$. Note that   $|M_i'\triangle M_i|\leq |M'\triangle M|\leq 147n^{1-\gamma}\frac{\log n}{\log{d}}\leq q n/30d^2$. It follows that by Lemma~\ref{aux:exp1toexp2}, that  for every $i=1,2,3$ and $j=1,2,\dots, 6$  the  pair $(D_i^j,M_i')$ is a  $(d, q^{-4}, 2q, {n/3})$-expander.
 
Setting $A=q^{-4}, \eps'=2q, \ell=8 \left \lceil\frac{\log(n/3)}{\log(d/4A)}\right \rceil\leq \frac{16\log n}{\log d}$ notice
 that we have $4d^{-1/2}\leq A^{-1}\leq\eps'/100\ll 1$ and $d \log d \geq 8A^2 \log(n/3)$.
So by Lemma~\ref{Lemma_expansion_paths}, for each $i=1,2, 3$ there is a subset $J_i\subseteq V(G_i)$  of size at most $4n/q^4d$ such that for  all  $x,y\in G_i\setminus  J_i$ lying in different parts of the bipartition of $G_i$, there is a rainbow $(D^{2i-1}_i\cup D^{2i}_i)$-$(M_i')$-alternating path of  length  at most $\ell$  from $x$ to $y$ in $V(G_i)$. To finish the proof we need the following two simple claims, whose statements are illustrated by Figure~\ref{Figure_switching}).

\begin{claim}\label{claim:forbiddenvertices}
There is an edge $x_1y_1'\in M_1'$ such that $ x_1, y_1' \not\in J_1$ and  $x_0y_1'\in E(H)$ such that either $c(x_0y_1')\not\in C(M')$ or there exists an edge $x_2y_2\in M_2'$ such that $c(x_0y_1')=c(x_2y_2)$ with $x_2, y_2\not\in J_2$.
\end{claim}
 \begin{proof}
 Recall that $|N_{C_2\cup C_0}(x_0)\cap V_1|\geq pn/40$ and so we have one of the following two options:
 \begin{enumerate}[(i)]
     \item  $|N_{C_0}(x_0)\cap V_1|\geq pn/80,$
     \item $|N_{C_2}(x_0)\cap V_1|\geq pn/80$,
 \end{enumerate}  

 \textbf{Case 1:} Let  $F(x_0)$  be the set of vertices $y_1'\in V_1$ satisfying one of the ``forbidden'' properties below.
 \begin{itemize} 
 \item [(F1)] $y_1'\in N_{C_0\cup C_2}(x_0)\cap V_1\setminus V(M_1')$. The number of these is at most $|V_1\setminus V(M_1')|\leq |V_1\setminus V(M_1)| + 2|M_1\triangle M_1'|\leq \frac{n^{1-\eps/8}}{3}+n^{1-\gamma}+ 294n^{1-\gamma}\frac{\log n}{\log{d}} \ll pn/80.$ 
\item [(F2)] $y_1'\in J_1$  or the vertex that $y_1$ is matched to in $M_1'$  is in $J_1$. The number of these is at most $|J_1|\leq  4n/q^4d \ll pn/80.$ 

\item [(F3)] $y_1'\in V_1$ such that $c(x_0y_1)\in C_0\cap C(M')$. Notice that this is possible as when we extended $M$ to $M'$ we potentially used some of the colours in $C_0$. However, the number of these is at most $|M'\setminus M|\leq 147n^{1-\gamma}\frac{\log n}{\log{d}}\ll pn/80.$
 \end{itemize}
 
Since $|F(x_0)|\ll |N_{C_0}(x_0)\cap V_1|$  we can pick a vertex $y_1\in N_{C_0}(x_0)\cap V_1$ not satisfying (F1)-(F3). Let $x_1$ be the vertex  that is matched to $y_1'$ in $M_1$.  It is easy to check that the following hold: $c(x_0y_1')\in C_0 \setminus C(M')$,  $x_1,y_1'\notin J_1$.
 
\textbf{Case 2:} In this case $F(x_0)$ will include the vertices $y_1'$ satisfying (F1) or (F2) and additionally the properties below.

 \begin{itemize} 
 \item [(F4)] $y_1'\in N_{C_2}(x_0)\cap V_1$ such that $c(x_0y_1')\in C_2\setminus C(M_2')$. The number of these is at most $|C_2\setminus C(M_2')|\leq |C_2\setminus C(M_2)| + |M_2\triangle M_2'|\leq \frac{n^{1-\eps/8}}{3}+n^{1-\gamma}+ 147n^{1-\gamma}\frac{\log n}{\log{d}} \ll pn/80.$
\item [(F5)] $y_1'\in N_{C_2}(x_0)\cap V_1$ such that $c(x_0y_1')\in C(M_2')$  but if we look at the edge $x_2y_2$ of $M_2'$ which has colour $c(x_0y_1')$ either $x_2\in J_2$ or $y_2\in J_2$. The number of these is at most $|J_2|\leq  4n/q^4d \ll pn/80.$ 
 \end{itemize}
 
Since $|F(x_0)|\ll |N_{C_2}(x_0)\cap V_1|$  we can pick a vertex $y_1'\in N_{C_2}(x_0)\cap V_1$ not satisfying (F1),(F2),(F4),(F5). Let $x_1$ be the vertex  that is matched to $y_1'$ in $M_1$ and let $x_2y_2$ be the edge in $M_2'$ of colour $c(x_0y_1)$ (by the choice of $y_1'$ such an edge exists).  Additionally, it is easy to check that the following hold: $x_1,y_1'\notin J_1$, $x_2,y_2\notin J_2$.
\end{proof}

\begin{claim}\label{claim:forbiddenvertices2}
There is an edge $x_1'y_1\in M_1'$ such that $ x_1'\neq x_1, y_1\neq y_1'$ and $x_1',y_1 \not\in J_1$ and  $y_0x_1'\in E(H)$ such that $c(y_0x_1')\neq c(x_0y_1')$ and  either $c(y_0x_1')\not \in C(M')$ or there exists an edge $x_3y_3\in M_3'$ such that $c(y_0x_1')=c(x_3y_3)$ with $x_3, y_3\not\in J_2$.
\end{claim}

\begin{proof}
The proof  is identical to the proof of Claim~\ref{claim:forbiddenvertices}, with extra conditions $x_1'\neq x_1$,$y_1'\neq y_1$, $c(y_0x_1')\neq c(x_0y_1')$ which  affect the calculations on $F(y_0)$ only by negligible amount. So we omit the proof.
\end{proof}

We may assume both Claim~\ref{claim:forbiddenvertices} and Claim~\ref{claim:forbiddenvertices2} hold with the second outcome, as otherwise the proof is even simpler (in that case $P_2$ or $P_3$ can be taken as empty, in the proof below). Let $x_1,y_1, x_1',y_1', x_2,y_2,x_3,y_3$ be as in the claims.  For each $i=1,2,3$  there is a  rainbow $(D_{2i-1}\cup D_{2i})$-$(M_i')$-alternating path  $P_i$ of  length  at most $\ell$  from $x_i$ to $y_i$ in $V_i$. Note that the paths $P_1$, $P_2$, $P_3$ don't share any vertices or colours. Finally let $M''$ be obtained from $M'$ by switching the matching edges along alternating cycles $P_2\cup \{x_2 y_2\}$ and $P_3\cup\{x_3y_3\}$, and by switching along the alternating path $P_1\cup\{x_0y_1',x_1y_1',x_1'y_0,x_1'y_1\}$,   (see Figure~\ref{Figure_switching}).  It is not hard to see that this is a rainbow matching with $|M''|=|M'|+1=|M|+r+1$ and such that 

\begin{figure}
  \centering
    \includegraphics[width=0.8\textwidth]{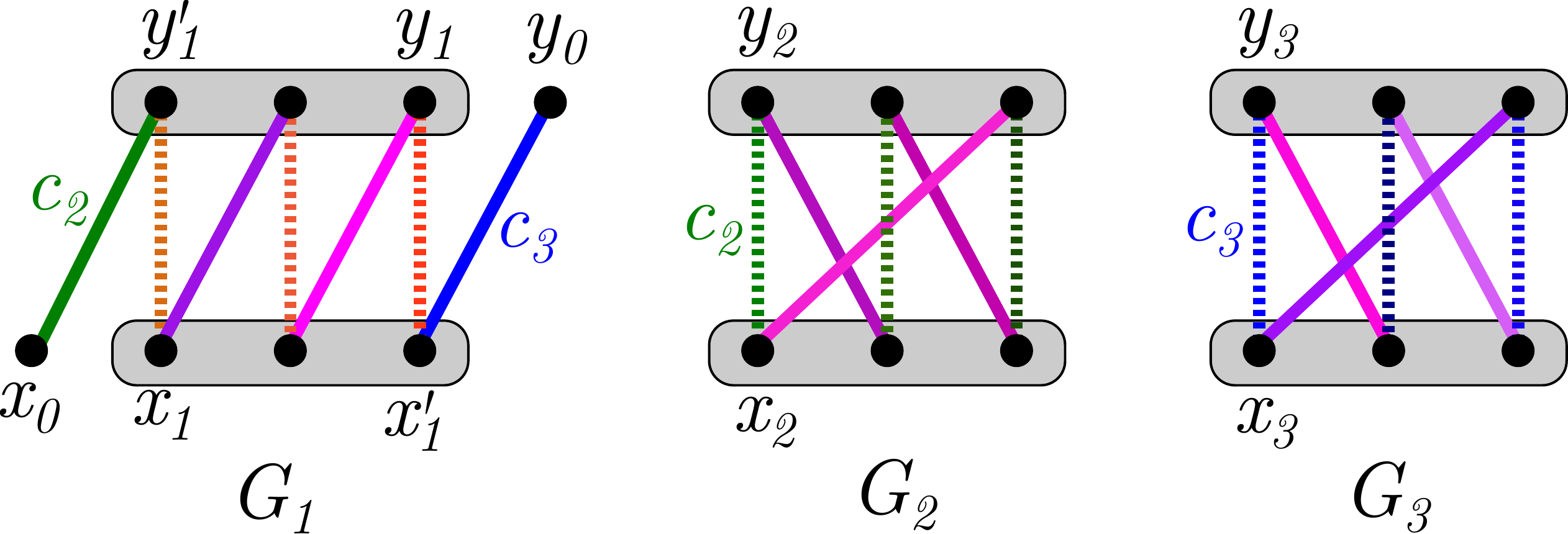}
  \caption{The alternating path $P_1\cup\{x_1y_1',x_1'y_1,x_0y_1',y_0x_1'\}$ and alternating cycles $P_2\cup \{x_2 y_2\}$ and $P_3\cup\{x_3y_3\}$. The dashed edges denote edges of $M'$ which are removed from the matching. The solid edges denote edges of colours in $\cup_{j=1}^6{D^j}\cup C_0$ which are added to the matching.}
\label{Figure_switching}
\end{figure}

$$|M''\triangle M|\leq |M'\triangle M| + 3 \ell+6\leq  49r\frac{\log n}{\log{d}} + 48\frac{\log{n}}{{\log d} }+6 \leq  49(r+1)\frac{\log n}{\log{ d}}.$$
This contradicts the maximality of  $M'$, therefore $M'$ must have been a perfect rainbow matching of $H$.   
\end{proof}

As a corollary of the above theorem, coloured typical graphs have rainbow matchings covering all but $\log n/\log\log n$ vertices.
\begin{corollary}\label{Corollary_typical_matching}
Let $n^{-1}\ll k^{-1}\ll p \leq 1$, $n^{-1}\ll \eps<1$ and fix $d=\frac{k\log n }{\log\log n}$.
 Let $G$ be coloured $(\eps, p, n)$-typical bipartite graph with parts of size $\geq n$ and at least $n$ colours. Then $G$ has a rainbow matching of size $n-6d$.
\end{corollary}

\begin{proof}
Let $X,Y$ be the parts of the bipartition of $G$. By the assumptions, we have $n\leq |X|,|Y|\leq n+n^{1-\eps}$.
Fix $n'=n-6d$, and delete vertices from each part of $G$ to get a balanced bipartite graph $G'$ with parts of size $n'$.

Notice that the number of vertices deleted from each part of $G$ is between $6d$ and $n^{1-\eps}+6d\ll n'^{1-\eps/2}$.   We claim that $G'$ is $(\eps/2,p,n')$-typical. Indeed, for every vertex $v\in V(G')$, $pn(1-n^{-\eps}) - n^{1-\eps} -6d \leq d_{G'}(v)\leq pn(1+n^{-\eps})$, thus $d_{G'}(v)=pn'(1\pm n'^{-\eps/2})$. Similarly, it can be shown that for any $u,v\in V(G')$, $d_{G'}(u,v)=p^2n'(1\pm n'^{-\eps/2})$, and that for every colours $c,c'$ we have $e_{G'}(c), e_{G'}(c')=pn'(1\pm {n'}^{-\eps/2})$ and $|V_{G'}(c)\cap  V_{G'}(c')\cap X|, |V_{G'}(c)\cap  V_{G'}(c')\cap Y|=p^2n'(1\pm {n'}^{-\eps/2})$. \Alexey{[Previously we were only checking the old definition of colour typical here]}
This implies that $G'$ is $(\eps/2, p, n')$-typical and has at least $n=n'+6d$ colours.

Thus the assumptions of  Theorem~\ref{Theorem_typical_technical} are satisfied with with $H$ and $G$ being the same graph, $G'$. It follows that $G'$ has a perfect rainbow matching, which induces a rainbow matching of size $n'=n-6d$ in $G$, as required. 
\end{proof}

We are now ready to reduce the proof of one of our main   results from Corollary~\ref{Corollary_typical_matching}.
\begin{proof}[Proof of Theorems~\ref{main-ryser} and~\ref{Theorem_Ryser_Rainbow_Version}]
As mentioned in Section~\ref{Section_proof_sketch}, Theorems~\ref{main-ryser} and~\ref{Theorem_Ryser_Rainbow_Version} are equivalent. So we will just prove Theorem~\ref{Theorem_Ryser_Rainbow_Version}.

We may assume $n\geq n_0$ for some implicit $n_0$ sufficiently large, as otherwise the theorem  is vacuously true for $k=\frac{n_0\log{\log{n_0}}}{\log{n_0}}$. So we choose $k$ such that $1\ll k \ll n_0$. 
Let $K_{n,n}$ be properly $n$-edge-coloured. Since every colour forms a perfect matching, we have that $K_{n,n}$ is coloured
$(1,1,n)$-typical with $\eps=1 \gg 1/n_0$. Thus, we can apply  Corollary~\ref{Corollary_typical_matching} to this $K_{n,n}$ and obtain a rainbow matching of size $n-k\log{n}/\log{\log{n}},$ as desired.
\end{proof}
 
\section{Transversals in Generalized Latin squares}
In this section we prove Theorem~\ref{main-manysymbols}. We  will use the following result of Pokrovskiy, Montgomery and Sudakov \cite{montgomery2019decompositions}.

\begin{theorem}\label{thm:PMS}There exists $\alpha>0$ such that for all $n^{-\alpha}/\alpha < \eps <1$ the following holds. If $K_{n,n}$ is properly edge coloured graph with at most $(1-\eps)n$ colours having more than $(1-\eps)n$ edges then $K_{n,n}$ has $(1-\eps)n$ edge disjoint perfect rainbow matchings.
\end{theorem}
We'll also use the following lemma giving small rainbow matchings in coloured bipartite graphs.

\begin{lemma}\label{Lemma_small_matching_min_degree}
Let $G$ be a properly edge-coloured balanced bipartite graph with $\delta(G)\geq d$, parts of size $n$, with every colour appearing at most $n/12$ times and such that $n\geq 3d+12$. Then $G$ has a rainbow matching of size at least  $3d/2$.
\end{lemma}

\begin{proof}
Let the parts of $G$ be $A,B$ with $|A|=|B|=n$.
Let $M_1$ be a maximum rainbow matching in $G$. For contradiction, assume $|M_1|<3d/2$. Let $A_1:=V(M_1)\cap A$, $B_1:=V(M_1)\cap B$, let $A_0:=A\setminus A_1$ and $B_0:=B\setminus B_1$.  
 
Let $B_1'$ be the set  of vertices in $B_1$ which have at least two edges of unused  colours going to  $A_0$. Similarly define $A_1'$. Note that if there is any edge $ab\in M_1$ such that $a\in A_1'$ and $b\in B_1'$ then we can get a larger matching by replacing $ab$ by different coloured edges from $a$  to $B_0$ and $b$ to $A_0$, thus contradicting the maximality of $M_1$. It follows that  $|M_1|\geq |A_1'|+ |B_1'|$.  By the minimum degree condition we have $e_G(A_0, B)\geq d |A_0|$. Note also that all edges of unused colours must be adjacent to $A_1$ or $B_1$. Let $\tilde{e}_G(A_0, B_1)$ be the number of edges going from $A_0$ to $B_1$ using only unused colours. Using that every colour occurs $\leq n/12$ times and $|M_1|=|B_1|$, we have
$$\tilde{e}_G(A_0, B_1)\geq e_G(A_0, B) -|M_1|n/12 \geq d|A_0|-|B_1|n/12.$$

On the other hand, from the definition of $B_1'$, $$\tilde{e}_G(A_0, B_1)\leq |A_0||B_1'|+ |B_1|.$$

Thus, we get that $|B_1'|\geq d- \frac{|B_1|}{|A_0|} (1+n/12)  \geq 3d/4$ (using $|B_1|\leq 3d/2$, and $|A_0|\geq n-3d/2$ and $n\geq 3d+12$). Similarly, $|A_1'|\geq  3d/4$. Therefore we get that $|M_1|\geq |A_1'|+|B_1'|\geq 3d/2$.
\end{proof}

Before proving Theorem~\ref{Theorem_ManySymbol_Rainbow_Version} we explain main ideas. The basic idea is to employ Theorem~\ref{Theorem_typical_technical}. Given a proper edge-colouring of $K_{n,n}$ with roughly $n\log{n}/\log{\log{n}}$ many colours, by Theorem~\ref{thm:PMS}, we may assume that at least $n-o(n)$ colours appear $n-o(n)$ times. Call these colours \emph{large}. 
The rest of the colours will be so-called \emph{small} colours. Note that small colours might even appear only once in the entire graph $K_{n,n}$. However, since there are many of them, that is, roughly  of order $n\log{n}/\log{\log{n}}$, we can greedily select a rainbow matching  $M_0$ of size at least $O(\frac{\log{n}}{\log{\log{n}}}) +t$ containing only small colours (this is done in Claim~\ref{claim:greedy}). Now look at the  graph obtained from  $K_{n,n}$ by keeping only the edges of large colours
and deleting the vertices of $M_0$ (note that in particular we exclude all the colours appearing on $M_0$). This graph might have some \emph{bad} vertices of low degree, since they were adjacent to many edges of small colours in  $K_{n,n}$. However, using the property that large colours appear $n-o(n)$  times and there are $n-o(n)$ many of them one can  prove that there are only few such bad vertices. So we can delete them as well and call the remaining graph $G$. We show that $G$ is coloured typical. Also note that $G$ contains $n-t=n-|M_0|+O(\frac{\log{n}}{\log{\log{n}}})$ large colours.
Finally let $H$ to be  the original $K_{n,n}$ minus the vertices and colours of $M_0$ removed. After checking that the graphs $H$ and $G$  satisfy all the conditions of Theorem~\ref{Theorem_typical_technical}
we obtain a rainbow matching $M$ in $H$ of size $|V(H)|=n-|M_0|$. Then $M_0\cup M$ is a rainbow matching of size exactly $n$ in $K_{n,n}$.

\begin{proof}[Proof of Theorem~\ref{Theorem_ManySymbol_Rainbow_Version}] We may assume $n\geq n_0$ for some sufficiently large $n_0$. As otherwise the theorem is true vacuously for  $k_0=\frac{n_0\log{\log{n_0}}}{\log_{n_0}}$. Choose $k$ such that $1\ll k \ll n_0$. Let $\alpha$ be derived from Theorem~\ref{thm:PMS}. Fix $d= \frac{k\log{n}}{72\log{\log {n}}}$  and $ \eps_0:=n^{-\alpha/2}$.  
  
Let $K_{n,n}$ be properly coloured with at least $72nd$ colours. By Theorem~\ref{thm:PMS}, we may assume that more than $(1-\eps_0)n$ colours have more than $(1-\eps_0)n$ edges. We call such colours \emph{large}. If a colour has less than $(1-\eps_0)n$ edges we call it  \emph{small}.

Choose $t$, so that the number of large colours is $n-t$. If the number of large colours is $>n$, then we instead fix $t=0$. This way $0\leq t\leq n^{1-\alpha/2}$ always holds.

\begin{claim}\label{claim:greedy}There exists a rainbow matching of small colour edges of size $t+6d$. 
\end{claim}
\begin{proof} If a small colour appears  less than $n/12$ times we call it \emph{tiny}, oherwise \emph{medium}. Let $m$ be the number of  medium colours.

If $m\geq t+6d$, then, using that $n/12 \geq 3(t+6d)$, we can greedily pick one edge per medium colour and obtain a rainbow matching of size $t+6d$. 
So we may assume the number of medium colours is less than  $t+6d$. As the total number of edges is $n^2$, the number of large colours must be less than $(1+2\eps_0)n$. Thus there are at least $kn\log{n}/\log\log n -  (1+2\eps_0)n-t-6d \geq kn\log{n}/2\log{\log{n}}$ tiny colours.  Furthermore, we may assume $m<t-12d$. Indeed, we can greedily pick a rainbow matching  $M_{\textit{tiny}}$ of tiny colours of size $18d$. We can do this because each edge in $M_{\textit{tiny}}$ forbids $2n$ edges which intersect it, so it forbids at most $2n$ tiny colours, but the number of available tiny colours is at least $kn\log{n}/2\log{\log{n}} = 36dn$. Thus, if $m \geq t-12d$, then we can find a rainbow matching $M_{\textit{medium}}$ of medium colours of size $t-12d$ greedily in $K_{n,n}\setminus V(M_{\textit{tiny}})$, since each  medium colour appears at least  $n/12 \geq 3(t+6d)$ times. Then $M_{\textit{medium}} \cup M_{\textit{tiny}}$ is a matching of small colours of size $t+6d$.

So we may assume  $m<t-12d$. Let $G$ be the subgraph of $K_{n,n}$ induced by tiny colours. Note that, by definition of $t$, every vertex is incident to at least $t$ edges of small colours. Since there are $m$ medium colours, we obtain that $\delta(G)\geq t-m \geq 12d$. We can apply Lemma~\ref{Lemma_small_matching_min_degree} to $G$ and obtain a rainbow matching  $M_{\textit{tiny}}$ of tiny colours of size at least $3(t-m)/2 > t-m+6d$, since $m < t-12d$. Again in $K_{n,n}\setminus V(M_{\textit{tiny}})$ we can greedily pick a rainbow matching  $M_{\textit{medium}}$ of size $m$ of medium colours, since each one of these colours appears at least $n/12 \geq 3(t+6d)$ times.  Taking $M_{\textit{medium}} \cup M_{\textit{tiny}}$ finishes the proof.
\end{proof}

Let $M_0$ be a rainbow matching of size $t+6d$ from the above claim.   Let $V_{small}$ be the set of vertices
which have more than $2\sqrt{\eps_0} n$ edges of small colours passing through them. Note that $|V_{small}|\leq 2\sqrt{\eps_0}n$ (otherwise, we would get more that $2\eps_0n^2$ small colour edges in the graph, contradicting ``more than $(1-\eps_0)n$ colours have more than $(1-\eps_0)n$ edges''). Let $G$ be obtained from $K_{n,n}$ by removing all edges of small colours  and all vertices from $V(M_0)\cup V_{small}$. 
Let $H$ be $K_{n,n}$ with all colours and vertices of $M_0$ removed. It is easy to see that $G\subset H$. Next we check that we can apply Theorem~\ref{Theorem_typical_technical}  to $G$ and $H$.

Notice that $H$ is balanced bipartite with parts of size $n'=n-t-6d$.  Notice that the parts in $G$ have size $\geq n'- |V_{\textit{small}}|\geq n'-2\sqrt{\eps_0}n = n'-2n^{1-\alpha/4}\geq n'(1-n'^{-\alpha/5})$ and also that the number of colours in $G$ is $\geq (1-\eps_0)n\geq n'(1-n'^{-\alpha/5})$. Every vertex $v\in V(H)$ satisfies $|N_H(v) \cap V(G)|\geq n-2|M_0|-|V_{small}|\geq n-2\sqrt{\eps_0}n -2(t+6d) \geq 0.3 n'$.

Next we show that $G$ is coloured $(\alpha/5, 1, n')$-typical. Using the fact that $G$ consists of edges of only large colours and that vertices in $G$ are adjacent to at most $2\sqrt{\eps_0} n$  small coloured edges in  the original graph $K_{n,n}$, we get that the following is true for any $u,v\in V(G)$ on the same side of $G$ and colours $c,c'\in C(G)$.

\begin{eqnarray*}
    d_G(v) &\geq& n - 2\sqrt{\eps_0} n -|V_{small}|-2|M_0|
    \geq (1-n'^{-\alpha/5})n'  \\
        d_G(u,v) &\geq& n - 4\sqrt{\eps_0} n -|V_{small}|-2|M_0|
    \geq (1-n'^{-\alpha/5})n' \\
    |E_G(c)|&\geq& (1-\eps_0)n-|V_{small}|-2|M_0| 
    \geq (1-n'^{-\alpha/5})n'  \\
    |V_G(c)\cap V_G(c')\cap X|&\geq& (1-2\eps_0)n-|V_{small}|-|M_0)|\geq (1-n'^{-\alpha/5})n' \\
    |V_G(c)\cap V_G(c')\cap Y|&\geq& (1-2\eps_0)n-|V_{small}|-|M_0|\geq (1-n'^{-\alpha/5})n' 
\end{eqnarray*}

Note that this is enough to conclude that all three graphs $G$, $G_{X,C}$ and $G_{Y,C}$ are (uncoloured) $(\alpha/5, 1, n')$-typical. Finally $G$ contains 
$n-t=n'+6d$ large colours. So $G$ and $H$ satisfy all the assumptions of Theorem~\ref{Theorem_typical_technical}, therefore we obtain  a perfect rainbow matching  $M$ in $H$ (whose colours, by definition, are disjoint from $M_0$). Finally, $M\cup M_0$ is a perfect rainbow matching in $K_{n,n}$.
\end{proof}

\section{Large matchings in Steiner systems}

In this section we improve the bound on Brouwer's conjecture about matchings in Steiner triple systems.
\begin{proof}[Proof of Theorem~\ref{main-steiner}] We may assume $n\geq n_0$ for $n_0$ sufficiently large. Choose ${n_0}^{-1}\ll k^{-1}\ll 1$ and fix $d=\frac{k\log n}{6\log \log{n}}$.
We assume $V(S)=\{x_1,x_2,\dots, x_n\}$. In a Steiner triple system we have $n\equiv 1$ or $3 \pmod 6$.
We'll first prove the result when   $n\equiv  3 \pmod 6$. In the other case, the same proof works if we apply it to a subgraph of $S$ formed by deleting a vertex. 
Let $G$ be  an auxiliary bipartite simple graph with bipartition $X=Y=V(S)$, colour set $V(S)$, and edge $ab$ having colour $c$ whenever $\{a,b,c\}\in E(S)$. Using that $S$ is a Steiner triple system, notice that $G$ is properly $n$-edge-coloured $K_{n,n}$ minus a perfect rainbow matching. It has codegrees $n-2$ and hence it is in particular coloured $(1-o(1), 1, n)$-typical.

We randomly construct a bipartite graph $H$ as follows. Partition $[n]$ into three disjoint sets  $I_X,I_Y,I_C$ by putting independently every $i$ in one of the sets $I_X,I_Y,I_C$ with probability $1/3$. Let $A=\{x_i:i\in I_X\}$, $B=\{x_i:i\in I_Y\}$, $C=\{x_i:i\in I_C\}$. Let $H$ be the subgraph of  $G$ consisting of edges from $A$ to $B$ having a colour in $C$. We claim that the  following simultaneously hold with positive probability:
\begin{enumerate}
    \item [(P1)] $H$ is properly coloured $(1/8, 1/3, n/3)$-typical.
    \item [(P2)]$|A|=|B|=|C|=n/3$.
\end{enumerate}

We show that (P1) holds with high probability and  (P2) holds  with positive probability, thus the claim will follow.

\textbf{(P1) holds with probability at least  $1-o(n^{-3})$.}

We can apply Lemma~\ref{Lemma_random_subgraph_typical_graph}  to $H$ as it is easy to check that it satisfies the assumptions in (b) with $q_{\ref{Lemma_random_subgraph_typical_graph}}=1/3$, $\eps_{\ref{Lemma_random_subgraph_typical_graph}}=1-o(1), p_{\ref{Lemma_random_subgraph_typical_graph}}=1$. 
Indeed, since $S$ is $3$-uniform, we have that for every edge $e\in G$ going through $x_i\in X, x_j\in Y, x_{\ell}\in C(G)$ the indices $i,j,\ell$ are distinct.
Thus Property (i) holds with probability at least $1-e^{-n^{1-\eps/2}} \geq  1-o(n^{-3})$.

\textbf{(P2) holds with probability at least $n^{-3}$.}

Notice that out of all the possible outcomes of the random variables $|A|,|B|,|C|$, the outcome  $|A|=|B|=|C|=n/3$ is the most likely one which happens with probability at least $1/n^3$. (The outcome $|A|=a, |B|=b, |C|=c$ has probability $\frac{1}{3^{n}} \binom{n}{a,b,c}$. For $n\equiv 0 \pmod{ 3}$ the multinomial coefficient $\binom{n}{a,b,c}$ is maximized when $a=b=c$).

Now we are ready to apply Corollary~\ref{Corollary_typical_matching}  to $H$.  We obtain a rainbow matching $M$ in $H$ of size at least $n/3-6d = n/3-k\log{n}/\log{\log{n}}$. Now it easy to see that the triples $M_S=\{(a,b, c(ab))| ab\in M\}$ induce a hypergraph matching in $S$ of the same size. Indeed, the fact that $(a,b, c(ab))$ is an edge of $S$ for all $ab\in M$ follows by  definition of $G$. To see that $M_S$ is a matching notice that for distinct edges $a_1b_1,a_2b_2\in M$, all four endpoints  $a_1,a_2,b_1,b_2$ correspond to four distinct vertices in $S$ because $(A,B,C)$ induce  a partition of $V(S)$ and $M$ is a matching. Finally $c(a_1b_1)\neq c(a_2b_2)$ since $M$ is rainbow and $c(a_1b_1),c(a_2b_2)$ are distinct from  $a_1,a_2,b_1,b_2$ since $(A,B,C)$ induce a partition of $V(S)$. 
\end{proof}

\section{Concluding remarks}
A far reaching generalisation of the Ryser-Brualdi-Stein conjecture was proposed in 1975 by Stein \cite{stein1975transversals}. He defined an equi-$n$-square as an $n\times n$ array filled with $n$ symbols such that every symbol appears \emph{exactly $n$ times}. Notice that Latin squares are equi-$n$-squares, but there are many equi-$n$-squares which are not Latin. Stein \cite{stein1975transversals} conjectured that all equi-$n$-squares contain a transversal of size $n-1$. If true, this would imply that Latin squares have size $n-1$ transversals.

Recently, the second and third author \cite{pokrovskiy2019counterexample} disproved Stein's Conjecture by constructing equi-$n$-squares without transversals of size $n-\log n/42$. On the other hand, our Theorem~\ref{main-ryser} gives transversals 
in Latin squares of size at least $n-O(\log n/\log \log n)$. Thus, combining these two results, we obtain a full separation 
between Latin and equi-$n$-squares. 

Despite being false, Stein's Conjecture remains one of the outstanding problems in the area. 
In particular it would be very interesting to determine whether it is true asymptotically i.e. is it true that every equi-$n$-square has a transversal of size $n-o(n)$.
Here, the best currently known result is due to Aharoni, Berger, Kotlar, and Ziv \cite{ABKZ}, who used topological methods, to show that  equi-$n$-squares always have a transversal of size at least $2n/3$.

All our results can be rephrased as results about finding large matchings in $3$-uniform hypergraphs. Here the results are about \emph{linear} 3-uniform hypergraphs i.e. ones where every pair of vertices are contained in at most one edge. As mentioned in the introduction, $n\times n$ Latin squares correspond to linear $n$-regular, $3$-partite, $3$-uniform hypergraphs with $n$ vertices in each part. 
Using the technique of R\"odl's nibble, there have been general results proved about finding large matchings in linear regular (and nearly-regular) hypergraphs. In particular Alon, Kim, and Spencer \cite{alon1997nearly} showed that linear $3$-uniform, $pn$-regular hypergraphs of order $n$ have matchings of size $n-O(n^{1/2}\log^{3/2}n)$. Our results show that if additionally a certain graph associated with the hypergraph is pseudorandom, then the matching can cover all but $O(\log/\log\log n)$ vertices. Specifically, for a $3$-uniform hypergraph $\mathcal H$, define its shadow $\partial \mathcal H$ to be the graph formed by replacing every edge of $\mathcal H$ by a triangle. Then Corollary~\ref{Corollary_typical_matching} is equivalent to the following.
 \begin{theorem}\label{Theorem_hypergraph_tripartite}
Let $n^{-1}\ll k^{-1}\ll p \leq 1$, $n^{-1}\ll \eps<1$. Let $\mathcal H$ be a $3$-uniform, tripartite linear hypergraph with partition $(V_1,V_2,V_3)$, $|V_1|=|V_2|=|V_3|=n$. Suppose that for all $i \neq j$, the induced subgraph  $\partial \mathcal H[V_i\cup V_j]$ of the shadow   between $V_i$ and $V_j$ is $(\eps, p, n)$-typical.
 Then $\mathcal H$ has a matching of size $n-\frac{k\log n }{\log\log n}$.
\end{theorem}
If we only assume $(\eps, p, n)$-regularity rather than $(\eps, p, n)$-typicality then the hypergraph above is nearly $pn$-regular and is only known to have a matching of size $n-n^{1-\gamma}$ (e.g. from Lemma~\ref{Lemma_variant_of_MPS_nearly_perfect_matching}). With the added typicality condition we get a much larger matching. For non-tripartite hypergraphs we can prove the following analogue.
 \begin{theorem}\label{Theorem_hypergraph_general}
Let $n^{-1}\ll k^{-1}\ll p\leq 1$, $n^{-1}\ll \eps<1$. Let $\mathcal H$ be a $3$-uniform linear hypergraph on $n$ vertices. Suppose that for vertex $v$ we have $|N_{\partial\mathcal H}(v)|=(1\pm n^{-\eps})pn$ and
for every pair of vertices $u,v$, $|N_{\partial\mathcal H}(v)|=(1\pm n^{-\eps})pn$ and $|N_{\partial\mathcal H}(u)\cap N_{\partial\mathcal H}(v)| =(1\pm n^{-\eps})p^2n$.
Then $\mathcal H$ has a matching of size $n-\frac{k\log n }{\log\log n}$.
\end{theorem}
This theorem is proved identically to Theorem~\ref{main-steiner}. Indeed, the only change that needs to be made is to observe that the graph $G$ constructed in that proof will be $(\eps,p,n)$-typical (rather than $(1-o(1),1,n)$-typical as in Theorem~\ref{main-steiner}). Due to applications to Latin squares and Steiner triple systems, it is worthwhile to study further the hypergraphs appearing in Theorems~\ref{Theorem_hypergraph_tripartite} and~\ref{Theorem_hypergraph_general}. In particular it would be interesting to determine if they  always have matchings of size $n-O(1)$ or not.

\bibliographystyle{abbrv}
\bibliography{Ryser}

\begin{thebibliography}{10}

\bibitem{ABKZ}
R.~Aharoni, E.~Berger, D.~Kotlar, and R.~Ziv.
\newblock On a conjecture of stein.
\newblock {\em Abh. Math. Semin. Univ. Hambg.}, 87:203--211, 2017.

\bibitem{akbari2004transversals}
S.~Akbari and A.~Alipour.
\newblock Transversals and multicolored matchings.
\newblock {\em Journal of Combinatorial Designs}, 12(5):325--332, 2004.

\bibitem{alon1997nearly}
N.~Alon, J.-H. Kim, and J.~Spencer.
\newblock Nearly perfect matchings in regular simple hypergraphs.
\newblock {\em Israel Journal of Mathematics}, 100(1):171--187, 1997.

\bibitem{AKS}
N.~Alon, M.~Krivelevich, and B.~Sudakov.
\newblock List coloring of random and pseudo-random graphs.
\newblock {\em Combinatorica}, 19:453--472, 1999.

\bibitem{barat2017transversals}
J.~Bar{\'a}t and Z.~L. Nagy.
\newblock Transversals in generalized latin squares.
\newblock {\em Ars Math. Contemp.}, 16(1):39--47, 2019.

\bibitem{best2018transversals}
D.~Best, K.~Hendrey, I.~M. Wanless, T.~E. Wilson, and D.~R. Wood.
\newblock Transversals in latin arrays with many distinct symbols.
\newblock {\em Journal of Combinatorial Designs}, 26(2):84--96, 2018.

\bibitem{brouwer1981size}
A.~E. Brouwer.
\newblock On the size of a maximum transversal in a steiner triple system.
\newblock {\em Canadian Journal of Mathematics}, 33(5):1202--1204, 1981.

\bibitem{brouwer1978lower}
A.~E. Brouwer, A.~de~Vries, and R.~Wieringa.
\newblock A lower bound for the length of partial transversals in a latin
  square.
\newblock {\em Nieuw Archief Voor Wiskunde}, 26(2):330--332, 1978.

\bibitem{brualdi1991combinatorial}
R.~A. Brualdi, H.~J. Ryser, et~al.
\newblock {\em Combinatorial matrix theory}, volume~39.
\newblock Springer, 1991.

\bibitem{drake1977maximal}
D.~A. Drake.
\newblock Maximal sets of latin squares and partial transversals.
\newblock {\em Journal of Statistical Planning and Inference}, 1(2):143--149,
  1977.

\bibitem{eberhard2020asymptotic}
S.~Eberhard, F.~Manners, and R.~Mrazovi{\'c}.
\newblock An asymptotic for the hall--paige conjecture.
\newblock {\em arXiv preprint arXiv:2003.01798}, 2020.

\bibitem{euler1782recherches}
L.~Euler.
\newblock Recherches sur un nouvelle esp{\'e}ce de quarr{\'e}s magiques.
\newblock {\em Verhandelingen uitgegeven door het zeeuwsch Genootschap der
  Wetenschappen te Vlissingen}, pages 85--239, 1782.

\bibitem{frankl1985near}
P.~Frankl and V.~R{\"o}dl.
\newblock Near perfect coverings in graphs and hypergraphs.
\newblock {\em European Journal of Combinatorics}, 6(4):317--326, 1985.

\bibitem{hall1955complete}
M.~Hall and L.~Paige.
\newblock Complete mappings of finite groups.
\newblock {\em Pacific Journal of Mathematics}, 5(4):541--549, 1955.

\bibitem{hatami2008lower}
P.~Hatami and P.~W. Shor.
\newblock A lower bound for the length of a partial transversal in a latin
  square.
\newblock {\em Journal of Combinatorial Theory, Series A}, 115(7):1103--1113,
  2008.

\bibitem{keevash2020number}
P.~Keevash and L.~Yepremyan.
\newblock On the number of symbols that forces a transversal.
\newblock {\em Combinatorics, Probability and Computing}, 29(2):234--240, 2020.

\bibitem{kim2002asymmetry}
J.~H. Kim, B.~Sudakov, and V.~H. Vu.
\newblock On the asymmetry of random regular graphs and random graphs.
\newblock {\em Random Structures \& Algorithms}, 21(3-4):216--224, 2002.

\bibitem{koksma1969lower}
K.~K. Koksma.
\newblock A lower bound for the order of a partial transversal in a latin
  square.
\newblock {\em Journal of Combinatorial Theory}, 7(1):94--95, 1969.

\bibitem{lindner1978note}
C.~C. Lindner and K.~T. Phelps.
\newblock A note on partial parallel classes in steiner systems.
\newblock {\em Discrete Mathematics}, 24(1):109--112, 1978.

\bibitem{molloy2013graph}
M.~Molloy and B.~Reed.
\newblock {\em Graph colouring and the probabilistic method}, volume~23.
\newblock Springer Science \& Business Media, 2013.

\bibitem{montgomery2019decompositions}
R.~Montgomery, A.~Pokrovskiy, and B.~Sudakov.
\newblock Decompositions into spanning rainbow structures.
\newblock {\em Proceedings of the London Mathematical Society},
  119(4):899--959, 2019.

\bibitem{pokrovskiy2019counterexample}
A.~Pokrovskiy and B.~Sudakov.
\newblock A counterexample to stein’s equi-$n$-square conjecture.
\newblock {\em Proceedings of the American Mathematical Society},
  147(6):2281--2287, 2019.

\bibitem{rodl1985packing}
V.~R{\"o}dl.
\newblock On a packing and covering problem.
\newblock {\em European Journal of Combinatorics}, 6(1):69--78, 1985.

\bibitem{ryser1967neuere}
H.~J. Ryser.
\newblock Neuere probleme der kombinatorik.
\newblock {\em Vortr{\"a}ge {\"u}ber Kombinatorik, Oberwolfach}, 69:91, 1967.

\bibitem{shor1982}
P.~W. Shor.
\newblock A lower bound for the length of a partial transversal in a latin
  square.
\newblock {\em Journal of Combinatorial Theory, Series A}, 33(3):1--8, 1982.

\bibitem{stein1975transversals}
S.~K. Stein.
\newblock Transversals of latin squares and their generalizations.
\newblock {\em Pacific Journal of Mathematics}, 59(2):567--575, 1975.

\bibitem{wang1978self}
S.~P. Wang.
\newblock {\em On self-orthogonal Latin squares and partial transversals of
  Latin squares}.
\newblock PhD thesis, The Ohio State University, 1978.

\bibitem{wilcox2009reduction}
S.~Wilcox.
\newblock Reduction of the {H}all--{P}aige conjecture to sporadic simple
  groups.
\newblock {\em Journal of Algebra}, 321(5):1407--1428, 2009.

\bibitem{woolbright1978n}
D.~E. Woolbright.
\newblock An $n\times n$ latin square has a transversal with at least $n-\sqrt{
  n}$ distinct symbols.
\newblock {\em Journal of Combinatorial Theory, Series A}, 24(2):235--237,
  1978.

\end{thebibliography}

\end{document}